\documentclass[12pt,a4paper]{amsart}
\usepackage{amsfonts}
\usepackage{amssymb}
\usepackage{amsmath}
\usepackage{amsthm}
\usepackage{cite}
\usepackage{enumitem}
\usepackage{tikz}
\usepackage{subfigure}

\theoremstyle{plain}
\newtheorem{thm}{Theorem}

\newtheorem{lem}{Lemma}
\newtheorem{prop}{Proposition}
\newtheorem{cor}{Corollary}
\newtheorem{bem}{Remark}

\newcommand{\vecII}[2]{
\ensuremath{
\begin{pmatrix}
#1 \\ #2 \\
\end{pmatrix}}}

\providecommand{\ov}{\overline}
\providecommand{\un}{\underline}   

\providecommand{\N}{\mathbb{N}}
\providecommand{\R}{\mathbb{R}}

\providecommand{\eps}{\varepsilon}
\providecommand{\ov}{\overline}
\providecommand{\dx}{\,dx}
\providecommand{\wto}{\rightharpoonup}
\providecommand{\skp}[2]{\langle#1,#2\rangle}

\DeclareMathOperator{\ind}{ind}

\DeclareMathOperator{\sech}{sech}
\DeclareMathOperator{\Id}{Id}
\renewcommand{\qed}{\hfill $\Box$}

\setitemize{itemsep=+2pt}
\setenumerate{itemsep=+2pt}

\begin{document}

\allowdisplaybreaks

\title[Minimal energy solutions and infinitely many bifurcating branches]{Minimal energy solutions and
infinitely many bifurcating branches for a class of saturated nonlinear Schr\"odinger systems}

\author{Rainer Mandel}
\address{R. Mandel \hfill\break
Scuola Normale Superiore di Pisa \hfill\break
I- Pisa, Italy}
\email{Rainer.Mandel@sns.it}
\date{\today}


\begin{abstract}
  We prove a conjecture which was recently formulated by Maia, Montefusco, Pellacci saying that minimal energy
  solutions of the saturated nonlinear Schr\"odinger system
  \begin{align*} 
    \begin{aligned}
    - \Delta u + \lambda_1 u &= \frac{\alpha u(\alpha u^2+\beta v^2)}{1+s(\alpha u^2+\beta v^2)} 
    \qquad\text{in }\R^n,\\
    - \Delta v + \lambda_2 v &= \frac{\beta v(\alpha u^2+\beta v^2)}{1+s(\alpha u^2+\beta v^2)}\qquad\text{in
    }\R^n,
    \end{aligned} 
  \end{align*}
  are necessarily semitrivial whenever $\alpha,\beta,\lambda_1,\lambda_2>0$ and
  $0<s<\max\{\frac{\alpha}{\lambda_1},\frac{\beta}{\lambda_2}\}$ except for the symmetric case
  $\lambda_1=\lambda_2,\alpha=\beta$. Moreover it is shown that for most parameter samples
  $\alpha,\beta,\lambda_1,\lambda_2$ there are infinitely many branches containing seminodal solutions which
  bifurcate from a semitrivial solution curve parametrized by $s$. 
\end{abstract}

\maketitle
\allowdisplaybreaks

\section{Introduction}
  
  In this paper we intend to continue the study on nonlinear Schr\"odinger systems for saturated optical
  materials which was recently initiated by Maia, Montefusco and Pellacci \cite{MMP_weakly_coupled}. In their
  paper the following system of elliptic partial differential equations
  \begin{align} \label{Gl equation}
    \begin{aligned}
    - \Delta u + \lambda_1 u &= \frac{\alpha u(\alpha u^2+\beta v^2)}{1+s(\alpha u^2+\beta v^2)} 
    \qquad\text{in }\R^n,\\
    - \Delta v + \lambda_2 v &= \frac{\beta v(\alpha u^2+\beta v^2)}{1+s(\alpha u^2+\beta v^2)} 
    \qquad\text{in }\R^n,
    \end{aligned} 
  \end{align}
  was suggested in order to model the interaction of two pulses within the optical material under
  investigation. Here, the parameters satisfy $\lambda_1,\lambda_2,\alpha,\beta,s>0$ and $n\in\N$. One way 
  to find classical fully nontrivial solutions of \eqref{Gl equation} is to use variational methods.
  The Euler functional $I_s:H^1(\R^n)\times H^1(\R^n)\to \R$ associated to the system \eqref{Gl equation} is
  given by
  \begin{align} \label{Gl Def I_s}
    \begin{aligned}
    I_s(u,v)
    &:=  \frac{1}{2}\Big(\|u\|_{\lambda_1}^2+\|v\|_{\lambda_2}^2
      -  \frac{\alpha}{s}\|u\|_2^2-\frac{\beta}{s}\|v\|_2^2\Big) 
      + \frac{1}{2s^2} \int_{\R^n} \ln(1+s(\alpha u^2+\beta z^2)) \\
    &=  \frac{1}{2}(\|u\|_{\lambda_1}^2+\|v\|_{\lambda_2}^2)
      - \frac{1}{2s^2} \int_{\R^n} g(sZ)
    \end{aligned} 
  \end{align}
  where $Z(x):= \alpha u(x)^2+\beta v(x)^2$ and $g(z):= z-\ln(1+z)$ for all $z\geq 0$. The symbol
  $\|\cdot\|_2$ denotes the standard norm on $L^2(\R^n)$ and the norms
  $\|\cdot\|_{\lambda_1},\|\cdot\|_{\lambda_2}$ are defined via 
  $$
    \|u\|_{\lambda_1} := \Big( \int_{\R^n} |\nabla u|^2+\lambda_1 u^2\Big)^{1/2},\qquad
    \|v\|_{\lambda_2} := \Big( \int_{\R^n} |\nabla v|^2+\lambda_2 v^2\Big)^{1/2}.
  $$ 
  %A nontrivial solution $(u,v)\in H^1(\R^n)\times H^1(\R^n)$ of \eqref{Gl equation} is called a ground
  %state if it has minimal energy $I_s(u,v)$ among all other nontrivial solutions. From elliptic regularity
  %theory one obtains that every ground state and even every finite energy solution of \eqref{Gl equation} is
  %classical.
  
  \medskip 
  
  Since we are interested in minimal energy solutions (i.e. ground states) for \eqref{Gl equation} the ground
  states $u_s,v_s$ of the scalar problems associated to \eqref{Gl equation} turn out to be of particular
  importance. These are positive radially symmetric and radially decreasing smooth functions satisfying
  \begin{equation} \label{Gl scalar GS}
    -\Delta u_s + \lambda_1 u_s = \frac{\alpha^2 u_s^3}{1+s\alpha u_s^2}\quad\text{in }\R^n,\qquad
    -\Delta v_s + \lambda_2 v_s = \frac{\beta^2 v_s^3}{1+s\beta v_s^2}\quad\text{in }\R^n.
  \end{equation}
  Since we will encounter these solutions many times let us recall some facts from the literature. 
  Existence of positive finite energy solutions $u_s,v_s$ of \eqref{Gl scalar GS} for parameters
  $0<s<\frac{\alpha}{\lambda_1}$ respectively $0<s<\frac{\beta}{\lambda_2}$ can be deduced from Theorem~2.2 in
  \cite{StuZho_Applying_the_mountain} in case $n\geq 3$ or Theorem~1(i) in
  \cite{GaSeTa_Existence_of} for $n\geq 2$. In case $n=1$ the positive functions $u_s,v_s$ are given
  by $u_s(x)=u_s(-x),v_s(x)=v_s(-x)$ for all $x\in\R$ and 
  \begin{align*}
    u_s|_{[0,\infty)}^{-1}(z) &= \int_z^{u_s(0)} \Big(\frac{1}{\lambda_1 x^2-s ^{-2}g(s\alpha
    x^2)}\Big)^{1/2}\,dx, \qquad \text{for }z\in (0,u_s(0)], \\
    v_s|_{[0,\infty)}^{-1}(z) &= \int_z^{v_s(0)} \Big(\frac{1}{\lambda_2 x^2-s ^{-2}g(s\beta
    x^2)}\Big)^{1/2}\,dx \qquad\quad \text{for } z\in (0,v_s(0)]
  \end{align*}
  where $u_s(0),v_s(0)>0$ are uniquely determined by
  \begin{equation} \label{Gl us(0) n=1}
    \lambda_1 u_s(0)^2 - s^{-2}g(s\alpha u_s(0)^2)
    = \lambda_2 v_s(0)^2 - s^{-2}g(s\beta v_s(0)^2)
    = 0.
  \end{equation}
  As in the explicit one-dimensional case it is known also in the higher-dimensional case that
  $u_s,v_s$ are radially symmetric, see Theorem~2 in \cite{GiNiNi_Symmetry}. Finally, the uniqueness of
  $u_s,v_s$ follows from Theorem~1 in \cite{SerTan_Uniqueness} in case $n\geq 3$ and from Theorem~1 in
  \cite{McLSer_Uniqueness} in case $n=2$. The uniqueness result for $n=1$ is a direct consequence of the
  existence proof we gave above.
    
  \medskip
  
  In this paper we strengthen the results obtained by Maia, Montefusco, Pellacci \cite{MMP_weakly_coupled}
  concerning ground state solutions and (component-wise) positive solutions of \eqref{Gl equation}, so let us
  shortly comment on their achievements. In Theorem~3.7 of their paper they proved the existence of
  nonnegative radially symmetric and nonincreasing ground state solutions of \eqref{Gl equation} for all
  $n\geq 2$ and parameter values $0<s<\max\{\frac{\alpha}{\lambda_1},\frac{\beta}{\lambda_2}\}$ where the upper
  bound for $s$ is in fact optimal by Lemma~3.2 in the same paper. It was conjectured that each of these
  ground states is semitrivial except for the special case $\alpha=\beta,\lambda_1=\lambda_2$ where the
  totality of ground state solutions is known in a somehow explicit way, see Theorem~2.1 in \cite{MMP_weakly_coupled} or
  Theorem~\ref{Thm Ground states}~(i) below. In \cite{MMP_weakly_coupled} this conjecture was proved for
  parameters $s\geq \min\{\frac{\alpha}{\lambda_1},\frac{\beta}{\lambda_2}\}$, see Theorem~3.15 and
  Theorem~3.17. Our first result shows that the full conjecture is true even in the case $n=1$ which was left
  aside in \cite{MMP_weakly_coupled}.
    
  \begin{thm} \label{Thm Ground states}
    Let $n\in\N, \alpha,\beta,\lambda_1,\lambda_2>0$ and
    $0<s<\max\{\frac{\alpha}{\lambda_1},\frac{\beta}{\lambda_2}\}$. Then the following holds:
    \begin{itemize}
      \item[(i)] In case $\alpha=\beta$ and $\lambda_1=\lambda_2$ all ground states of \eqref{Gl
      equation} are given by $(\cos(\theta)u_s,\sin(\theta)v_s)$ for $\theta\in [0,2\pi)$.
      \item[(ii)] In case $\alpha\neq\beta$ or $\lambda_1\neq \lambda_2$ every ground state solution of
      \eqref{Gl equation} is semitrivial.
    \end{itemize}
  \end{thm}
  
  The proof of this result will we presented in section~\ref{sec Proof of Thm1}. Our approach is based on a
  suitable min-max characterization of the Mountain pass level associated to \eqref{Gl equation} involving a
  fibering map technique as in \cite{Man_Minimal_I}. This method even allows to give an alternative proof
  for the existence of a ground state solution of \eqref{Gl equation} which is significantly shorter than the
  one presented in \cite{MMP_weakly_coupled} and which moreover incorporates the case $n=1$, see
  Proposition \ref{Prop cs}. More importantly this approach yields the optimal result. 
  
  \medskip
  
  In view of Theorem \ref{Thm Ground states} it is natural to ask how the existence
  of fully nontrivial solutions of \eqref{Gl equation} can be proved. In \cite{MMP_weakly_coupled} Maia,
  Montefusco, Pellacci found necessary conditions and sufficient conditions for the existence of positive solutions
  of \eqref{Gl equation} which, however, partly contradict each other. For instance, Theorem~3.21 
  in~\cite{MMP_weakly_coupled} claims that positive solutions exist for parameters $\alpha=\beta,\lambda_1\neq
  \lambda_2$ and $s>0$ sufficiently small contradicting the nonexistence result from Theorem~3.10. The error
  leading to this contradiction is located on page 338, line~13 in~\cite{MMP_weakly_coupled} where the number
  $\frac{\lambda_2}{s}$ must be replaced by $\lambda_2 s$ which destroys the results from Theorem~3.19 and
  Theorem~3.21. Our approach to finding positive solutions and, more generally, seminodal solutions of
  \eqref{Gl equation} is to apply bifurcation theory to the semitrivial solution branches
  \begin{align*}
    \mathcal{T}_1 := \Big\{ (0,v_s,s) : 0<s<\frac{\beta}{\lambda_2} \Big\},\qquad
    \mathcal{T}_2 := \Big\{ (u_s,0,s) : 0<s<\frac{\alpha}{\lambda_1} \Big\}     
  \end{align*}
  which was motivated by the papers of Ostrovskaya, Kivshar \cite{OstKiv_Multi-hump} and Champneys, Yang
  \cite{ChaYan_A_scalar_nonlocal}. In the case $n=1$ and $\lambda_1=1,\lambda_2=\omega^2\in
  (0,1),\alpha=\beta=1$ they numerically detected a large number of solution branches emanating from
  $\mathcal{T}_2$ which consist of seminodal solutions. Moreover, they conjectured that the bifurcation points
  on $\mathcal{T}_2$ accumulate near $s=1$, see page~2184~ff. in \cite{ChaYan_A_scalar_nonlocal}. Our results
  confirm these observations. For simplicity we will only discuss the bifurcations from $\mathcal{T}_2$
  since the corresponding analysis for $\mathcal{T}_1$ is the same up to interchanging the roles of
  $\lambda_1,\lambda_2$ and $\alpha,\beta$. Investigating the linearized problems associated to \eqref{Gl
  equation} near $(u_s,0,s)$ for parameters close to the boundary of the parameter interval
  $(0,\frac{\alpha}{\lambda_1})$ we prove the existence of infinitely many bifurcating branches containing
  fully nontrivial solutions of a certain nodal pattern. Despite the fact that the question whether fully
  nontrivial solutions bifurcate from $\mathcal{T}_1,\mathcal{T}_2$ makes perfect sense for all space
  dimensions $n\in\N$ our bifurcation result is restricted to $n\in\{1,2,3\}$. Later we will comment on this issue in more
  detail, see Remark~\ref{Bem ngeq4}. In order to formulate our bifurcation result let us define the positive
  numbers $\bar\mu_k$ to be the $k$-th eigenvalues of the linear compact self-adjoint operators
  $\phi\mapsto (-\Delta+\lambda_2)^{-1}(\alpha\beta u_0^2\phi)$ mapping $H^1_r(\R^n)$ to itself where
  $u_0$ denotes the positive ground state solution of the first equation in \eqref{Gl scalar GS} for $s=0$. 
  %Notice that these
  %functions only exist for $n\in\{1,2,3\}$ which follows from a direct application of Poho\v{z}aev's
  %identity. 
  By Sturm-Liouville theory we know that these eigenvalues are simple and that they satisfy 
  $$
    \bar\mu_0>\bar\mu_1>\bar\mu_2>\ldots>\bar\mu_k \to 0^+
    \quad\text{as }k\to\infty.
  $$
  Deferring some more or less standard notational convention to a later moment we come to the
  statement of our result.
  
  \begin{thm} \label{Thm Bifurcation}
    Let $n\in \{1,2,3\}$ and let $\alpha,\beta,\lambda_1,\lambda_2>0$ and $k_0\in\N_0$ satisfy
    $$
      \frac{\lambda_2}{\lambda_1}<\frac{\beta}{\alpha} \qquad\text{and}\qquad 
      \ov{\mu}_{k_0}<1.
    $$
    Then there is an increasing sequence $(s_k)_{k\geq k_0}$ of positive numbers converging to
    $\frac{\alpha}{\lambda_1}$ such that continua $\mathcal{C}_k\subset \mathcal{S}$ containing $(0,k)-$nodal
    solutions of \eqref{Gl equation} emanate from $\mathcal{T}_2$ at $s=s_k\,(k\geq k_0)$. In case
    $k_0=0$ we have $\lambda_1>\lambda_2$ and there is a $C>0$ such that all positive solutions
    $(u,v,s)\in\mathcal{C}_0$ with $s\geq 0$ satisfy
    \begin{equation} \label{Gl Thm bifurcation positive solutions}
      \|u\|_{\lambda_1}+\|v\|_{\lambda_2}<C\qquad\text{and}\qquad
      s<\frac{\alpha-\beta}{\lambda_1-\lambda_2}<\frac{\alpha}{\lambda_1}.
    \end{equation}    
  \end{thm}
  
  In case $n\in\{2,3\}$ we can estimate $\bar\mu^0$ from above in order to obtain a
  sufficient condition for the conclusions of Theorem~\ref{Thm Bifurcation} to hold for $k_0=0$. This estimate
  leading to Corollary~\ref{Cor 1} is based on the Courant-Fischer min-max-principle and H\"older's inequality.
  In the one-dimensional case the values of all eigenvalues $\bar\mu_k$ are explicitly known which results in
  Corollary~\ref{Cor 2}. 
  
  \begin{cor} \label{Cor 1}
    Let $n\in\{2,3\}$. Then the conclusions from Theorem~\ref{Thm Bifurcation} are true for $k_0=0$
    if
    \begin{equation} \label{Gl cond Cor1}
      \frac{\lambda_2}{\lambda_1}<\frac{\beta}{\alpha}<
      \Big(\frac{\lambda_2}{\lambda_1}\Big)^{\frac{4-n}{4}}.
    \end{equation}  
  \end{cor}

  \begin{cor} \label{Cor 2}
    Let $n=1$. Then the conclusions from Theorem~\ref{Thm Bifurcation} are true in case
    \begin{align} \label{Gl cond Cor2}
      \frac{\lambda_2}{\lambda_1}<\frac{\beta}{\alpha}<
      \frac{1}{2}\Big(\sqrt{\frac{\lambda_2}{\lambda_1}}+2k_0\Big)\Big(\sqrt{\frac{\lambda_2}{\lambda_1}}+2k_0+1\Big).
    \end{align}
  \end{cor}
  
  \begin{bem}     
    As we mentioned above one can find sufficient criteria for the existence of $(k,0)$-nodal solutions
    bifurcating from $\mathcal{T}_1$ by inversing the roles of $\lambda_1,\lambda_2$ and $\alpha,\beta$ in
    the statement of Theorem~\ref{Thm Bifurcation} as well as in its Corollaries.
  \end{bem}

  Theorem~\ref{Thm Bifurcation} gives rise to many questions which would be interesting to solve in the
  future. A list of open problems is provided in section~\ref{sec Open problems}. Before going on with the
  proof of our results let us clarify the notation which we used in Theorem~\ref{Thm Bifurcation}. The set
  $\mathcal{S}\subset X\times\R$ is the closure of all solutions of \eqref{Gl equation} which
  do not belong to $\mathcal{T}_2$ and a subset of $\mathcal{S}$ is called a
  continuum if it is a maximal connected set within $\mathcal{S}$. Finally, a fully nontrivial solution
  $(u,v)$ of \eqref{Gl equation} is called $(k,l)$-nodal if both component functions are radially symmetric
  and $u$ has precisely $k+1$ nodal annuli and $v$ has precisely $l+1$ nodal annuli. In other words, since
  double zeros can not occur, $(u,v)$ is $(k,l)$ nodal if the radial profiles of $u$ respectively $v$ have
  precisely $k$ respectively $l$ zeros.

  \section{Proof of Theorem \ref{Thm Ground states}} \label{sec Proof of Thm1}
  
  According to the assumptions of Theorem \ref{Thm Ground states} we will assume throughout this section that
  the numbers $\lambda_1,\lambda_2,\alpha,\beta$ are positive, that $s$ lies between $0$ and
  $\max\{\frac{\alpha}{\lambda_1},\frac{\beta}{\lambda_2}\}=:s^*$ and that the space dimension is
  an arbitrary natural number. Furthermore, we define the energy levels  
  \begin{align*}
    c_s &= \inf \Big\{ I_s(u,v) : (u,v)\in H^1(\R^n)\times H^1(\R^n) \text{ solves } \eqref{Gl equation},
    (u,v)\neq (0,0) \Big\}, \\
    c_s^* &= \inf \Big\{ I_s(u,v) : (u,v)\in H^1(\R^n)\times H^1(\R^n) \text{ solves } \eqref{Gl equation},
    u=0,v\neq 0 \text{ or }u\neq 0,v=0 \Big\}.
  \end{align*}  
  The first step towards the proof of Theorem \ref{Thm Ground states} is a more suitable
  min-max-characterization of the least energy level $c_s$ of \eqref{Gl equation} which, as in
  \cite{Man_Minimal_I}, gives rise to a simple proof for the existence of a ground state. To this end we
  introduce the Nehari manifold
  $$
    c_{\mathcal{N}_s} := \inf_{\mathcal{N}_s} I_s,\qquad
    \mathcal{N}_s := \{ (u,v)\in H^1(\R^n)\times H^1(\R^n): (u,v)\neq (0,0) \text{ and }I_s'(u,v)[(u,v)]=0\}.
  $$

  \begin{prop} \label{Prop cs} 
    The value   
    \begin{equation} \label{Gl cs-Char}
      c_s = c_{\mathcal N_s} = \inf_{(u,v)\neq (0,0)} \sup_{r>0} I_s(\sqrt{r}u,\sqrt{r}v).
    \end{equation}
    is attained at a radially symmetric and radially noninreasing ground state of \eqref{Gl equation}.
  \end{prop}
  \begin{proof}
    From the equations (3.15),(3.52) in \cite{MMP_weakly_coupled} we get $c_s=c_{\mathcal N_s}$, so let us
    prove the second equation in \eqref{Gl cs-Char}. For every fixed $u,v\in H^1(\R^n)$ satisfying $(u,v)\neq (0,0)$
    we set 
    $$
      \beta(r)
      := I_s(\sqrt{r}u,\sqrt{r}v)
      = \frac{r}{2}(\|u\|_{\lambda_1}^2+\|v\|_{\lambda_2}^2) - \frac{1}{2s^2}\int_{\R^n} g(rsZ)
    $$ 
    so that $(\sqrt{r}u,\sqrt{r}v)\in\mathcal{N}_s$ holds for $r>0$ if and only if $\beta'(r)=0$. Since
    $\beta$ is smooth and strictly concave with $\beta'(0)>0$ a critical point of $\beta$ is uniquely
    determined and it is a maximizer (whenever it exists). Since the supremum of $\beta$ is $+\infty$ when
    there is no maximizer of $\beta$ we obtain 
    $$
      c_{\mathcal N_s}
      = \inf_{\mathcal{N}_s} I_s
      = \inf_{(u,v)\neq (0,0)} \sup_{r>0} I_s(\sqrt{r}u,\sqrt{r}v)
    $$
    which proves the formula \eqref{Gl cs-Char}.
    
    \medskip  
    
    Due to $0<s<\max\{\frac{\alpha}{\lambda_1},\frac{\beta}{\lambda_2}\}$
    we can find a semitrivial function $(u,v)\in H^1(\R^n)\times H^1(\R^n)$ satisfying
    $\|u\|_{\lambda_1}^2+\|v\|_{\lambda_2}^2<\tfrac{\alpha}{s}\|u\|_2^2+\frac{\beta}{s}\|v\|_2^2$ which
    implies $c_s<\infty$ according to \eqref{Gl cs-Char}.
    So let $(u_k,v_k)$ be a minimizing sequence in $H^1(\R^n)\times H^1(\R^n)$  
    satisfying $\sup_{r>0} I_s(\sqrt{r}u_k,\sqrt{r}v_k)\to c_s$ as $k\to\infty$. Using the classical
    Polya-Szeg\"o-inequality and the extended Hardy-Littlewood inequality 
    $$
     \int_{\R^n} \ln \Big(1+ rs(\alpha u_k^2+\beta v_k^2)\Big)
     \geq  \int_{\R^n} \ln \Big( 1+rs(\alpha {u_k^*}^2+\beta {v_k^*}^2)\Big)
     \qquad\text{for all }r>0
    $$
    for the spherical rearrangement taken from Theorem 2.2 in  \cite{AlmLie_symmetric_decreasing} we may
    assume $u_k,v_k$ to be radially symmetric and radially decreasing. Since the function
    $g(z)=z-\ln(1+z)$ strictly increases on $(0,\infty)$ from 0 to $+\infty$ we may moreover assume that
    $(u_k,v_k)$ are rescaled in such a way that the equality $\frac{1}{2s^2}\int_{\R^n} g(sZ_k) = 1$
    holds for $Z_k:=\alpha u_k^2+\beta v_k^2$. The inequality 
    $$
      c_s + o(1)
      = \lim_{k\to\infty} \sup_{r>0} I_s(\sqrt{r}u_k,\sqrt{r}v_k)
      \geq \limsup_{k\to\infty} I_s(u_k,v_k)
      = \frac{1}{2} \limsup_{k\to\infty} (\|u_k\|_{\lambda_1}^2+\|v_k\|_{\lambda_2}^2) - 1
	$$ 
	implies that the sequence $(u_k,v_k)$ is bounded in $H^1(\R^n)\times H^1(\R^n)$. Using the uniform decay rate
	and the resulting compactness properties of radially decreasing functions bounded in
	$H^1(\R^n)\times H^1(\R^n)$ (apply for instance Compactness Lemma~2 in \cite{Str_existence_of_solitary})
	we may take a subsequence again denoted by $(u_k,v_k)$ such that $(u_k,v_k)\wto (u,v)$ in $H^1(\R^n)\times
	H^1(\R^n)$, pointwise always everywhere and 
	$$
	  \frac{1}{2s^2}\int_{\R^n} g(rsZ)
	  = \lim_{k\to\infty} \frac{1}{2s^2} \int_{\R^n} g(rsZ_k) \qquad\text{for all }r>0.
	$$
    From this we infer $\frac{1}{2s^2}\int_{\R^n} g(sZ)=1$ and thus $(u,v)\neq (0,0)$.
    Hence, we obtain
    \begin{align*}
      c_s 
      &= \lim_{k\to\infty} \sup_{\rho >0} I_s(\sqrt{\rho}u_k,\sqrt{\rho}v_k) \\
      &\geq  \limsup_{k\to\infty} \Big( \frac{r}{2} (\|u_k\|_{\lambda_1}^2+\|v_k\|_{\lambda_2}^2 ) 
      - \frac{1}{2s^2}\int_{\R^n} g(r sZ_k)\Big)\\
	  &\geq \frac{r}{2} (\|u\|_{\lambda_1}^2+\|v\|_{\lambda_2}^2) - \frac{1}{2s^2} \int_{\R^n} g(r sZ) \\
      &= I_s(\sqrt{r}u,\sqrt{r}v)
      \qquad \qquad \text{for all }r>0 
    \end{align*}
    so that $(u,v)$ is a nontrivial radially symmetric and radially decreasing minimizer.
    Taking for $r$ the maximizer of the map $r\mapsto I_s(\sqrt{r}u,\sqrt{r}v)$ we obtain the ground state
    solution $(\bar u,\bar v):=(\sqrt{r }u,\sqrt{r}v)$ having the properties we claimed to hold. Indeed,
    the Nehari manifold may be rewritten as $\mathcal{N}_s= \{ (u,v)\in H^1(\R^n)\times H^1(\R^n):
    (u,v)\neq (0,0), H(u,v)=0\}$ for  
    $$
      H(u,v)
      :=I_s'(u,v)[(u,v)]
	  = \|u\|_{\lambda_1}^2+\|v\|_{\lambda_2}^2 - \int_{\R^n}  \frac{Z^2}{1+sZ} 
    $$ 
    so that the Lagrange multiplier rule applies due to
    $$ 
      H'(u,v)[(u,v)]
      = 2(\|u\|_{\lambda_1}^2+\|v\|_{\lambda_2}^2) - \int_{\R^n} \frac{4Z^2+2sZ^3}{(1+sZ)^2}
      = - \int_{\R^n} \frac{2Z^2}{1+sZ}
      <0
    $$
    for all $(u,v)\in\mathcal{N}_s$.
%     From the existence and uniqueness of the semitrivial minimizer on $[0,S_4)$ we obtain the continuity
%     of the map $s\mapsto c_s$ since the radially symmetric decreasing ground states $(u_s,0)$ or $(0,v_s)$
%     converge in $H^1(\R^n)\times H^1(\R^n)$ as $s$ varies in $[0,S_4)$. Notice that $\|(u,v)\|_E$ is bounded
%     from below on $\mathcal{N}$ so that the limit solutions must be nontrivial as well. Moreover, $c_s$ is
%     increasing with respect to $s$ since we have for all $(u,v)\neq (0,0)$
%     \begin{align*}
%       \frac{d}{ds} I_s(u,v)
%       &= \frac{d}{ds} \Big( - \frac{1}{2s} \int_{\R^n} Z(x)\,dx + \frac{1}{2s^2} \int_{\R^n}
%       \ln(1+sZ(x))\,dx\Big)   \\ 
%       &= \frac{1}{2s^3} \int_{\R^n} \Big(sZ(x) + \frac{sZ(x)}{1+sZ(x)} - 2\ln(1+sZ(x))  \Big) \,dx\\ 
%       &= \frac{1}{2s^3} \int_{\R^n} \int_0^{sZ(x)} \Big( \frac{t}{1+t}\Big)^2\,dt \,dx \\
%       &> 0.
%     \end{align*}
  \end{proof}
  
  Let us note that $c_s$ equals $c=m_{\mathcal{N}}=m_{\mathcal{P}}$ from Lemma~3.6 
  in~\cite{MMP_weakly_coupled} and therefore corresponds to the Mountain pass level of $I_s$. 
  Given Proposition~\ref{Prop cs} we are in the position to prove Theorem~\ref{Thm Ground states}.

  \medskip
  
  {\it Proof of Theorem \ref{Thm Ground states}}:\; Part (i) was proved in
  \cite{MMP_weakly_coupled}, Lemma 3.2, so let us prove (ii). First we show that the ground state energy
  level $c_s$ equals $c_s^*$. Since we have $c_s\leq c_s^*$ by definition we
  have to show
    \begin{equation} \label{Gl Thm 1 I}
      \sup_{r>0} I_s(\sqrt{r}u,\sqrt{r}v)\geq c_s^* \qquad\text{for all }u,v\in H^1(\R^n) 
      \text{ with } u,v\neq 0.
    \end{equation}
    From \eqref{Gl Def I_s} we deduce the following: if $\|u\|_{\lambda_1}^2\geq
    \frac{\alpha}{s}\|u\|_2^2$ then we have $I_s(\sqrt{r}u,\sqrt{r}v)\geq I_s(0,\sqrt{r}v)$ for all $v\neq 0$
    and $r>0$ which implies the inequality \eqref{Gl Thm 1 I}. The same way one proves \eqref{Gl Thm 1 I} in case
    $\|v\|_{\lambda_2}^2\geq \frac{\beta}{s} \|v\|_2^2$ so that it remains to prove \eqref{Gl Thm 1 I} for
    functions $(u,v)$ satisfying 
    \begin{equation} \label{Gl Thm 1 II}
       \|u\|_{\lambda_1}^2<\frac{\alpha}{s}\|u\|_2^2\qquad\text{and}\qquad
       \|v\|_{\lambda_2}^2< \frac{\beta}{s} \|v\|_2^2.
    \end{equation}
    To this end let $r>0$ be arbitrary but fixed. From  \eqref{Gl Thm 1 II} we infer that the numbers    
    \begin{align*}
      t(u,v)
      &:=   \frac{\frac{\alpha}{s}\|u\|_2^2-\|u\|_{\lambda_1}^2}{
      \frac{\alpha}{s}\|u\|_2^2+\frac{\beta}{s}\|v\|_2^2 -\|u\|_{\lambda_1}^2 -\|v\|_{\lambda_2}^2},\\
      r(u,v) &:=
      r\cdot
      \Big(\frac{\alpha}{s}\|u\|_2^2+\frac{\beta}{s}\|v\|_2^2-\|u\|_{\lambda_1}^2-\|v\|_{\lambda_2}^2\Big)
    \end{align*}
    satisfy $t(u,v)\in (0,1),r(u,v)>0$ as well as 
    \begin{align}  \label{Gl Thm 1 III}
      \begin{aligned}
       I_s(\sqrt{r}u,\sqrt{r}v) 
       &=  -\frac{r(u,v)}{2} + \frac{1}{2s^2} \int_{\R^n} \ln\Big( 1+\frac{r(u,v) s(\alpha u^2+\beta v^2)}{
       \frac{\alpha}{s}\|u\|_2^2+\frac{\beta}{s} \|v\|_2^2-\|u\|_{\lambda_1}^2-\|v\|_{\lambda_2}^2} \Big).
      \end{aligned}
    \end{align}
     The concavity of $\ln$ yields
     \begin{align*}
        &\; \int_{\R^n} \ln\Big( 1+\frac{r(u,v) s(\alpha u^2+\beta v^2)}{
        \frac{\alpha}{s}\|u\|_2^2+\frac{\beta}{s} \|v\|_2^2-\|u\|_{\lambda_1}^2-\|v\|_{\lambda_2}^2} \Big)  \\
       &= \int_{\R^n} \ln\Big( t(u,v)
          \Big(1+\frac{r(u,v)s\alpha u^2}{\frac{\alpha}{s}\|u\|_2^2-\|u\|_{\lambda_1}^2}\Big)
        + (1-t(u,v))\Big(1+\frac{r(u,v)s\beta v^2}{\frac{\beta}{s}\|v\|_2^2-\|v\|_{\lambda_2}^2}
        \Big)\Big) \\
       &\geq t(u,v) \int_{\R^n} \ln\Big( 1+\frac{r(u,v)s\alpha
       u^2}{\frac{\alpha}{s}\|u\|_2^2-\|u\|_{\lambda_1}^2}\Big)
        + (1-t(u,v)) \int_{\R^n} \ln\Big(1+\frac{r(u,v)s\beta  
        v^2}{\frac{\beta}{s}\|v\|_2^2-\|v\|_{\lambda_2}^2} \Big)    \\
       &\geq \min\Big\{ \int_{\R^n} \ln\Big( 1+\frac{r(u,v) s\alpha
       u^2}{\frac{\alpha}{s}\|u\|_2^2-\|u\|_{\lambda_1}^2}\Big),\int_{\R^n} \ln\Big(1+\frac{r(u,v) s\beta  
        v^2}{\frac{\beta}{s}\|v\|_2^2-\|v\|_{\lambda_2}^2} \Big)   \Big\}.              
     \end{align*}
     Combining this inequality with \eqref{Gl Thm 1 III} gives
     \begin{align*}
       I_s(\sqrt{r}u,\sqrt{r}v) 
       &\geq \min\Big\{ - \frac{r(u,v)}{2}+\int_{\R^n} \ln\Big( 1+\frac{r(u,v) s\alpha
       u^2}{\frac{\alpha}{s}\|u\|_2^2-\|u\|_{\lambda_1}^2}\Big), \\
      &\qquad\qquad - \frac{r(u,v)}{2}+\int_{\R^n} \ln\Big(1+\frac{r(u,v) s\beta  
        v^2}{\frac{\beta}{s}\|v\|_2^2-\|v\|_{\lambda_2}^2} \Big)  
       \Big\}.
     \end{align*}
     Taking the supremum with respect to $r>0$ we get~\eqref{Gl Thm 1 I} and therefore $c_s\geq c_s^*$ which
     is what we had to show.
     
     \medskip
     
     It remains to prove that every ground state is semitrivial unless $\lambda_1=\lambda_2,\alpha=\beta$.
     To this end assume that $(u,v)$ is a fully nontrivial ground state solution of \eqref{Gl equation} so
     that in particular $I_s(u,v)=c_s$ holds. Then $c_s=c_s^*$ implies
     that the inequalities from above are equalities for some $r>0$. In particular, since $\ln$ is
     strictly concave and $t(u,v)\in (0,1)$ we get 
     $$
       1+\frac{r(u,v)s\alpha u^2}{\frac{\alpha}{s}\|u\|_2^2-\|u\|_{\lambda_1}^2} = k\cdot
       \Big(1+\frac{r(u,v)s\beta v^2}{\frac{\beta}{s}\|v\|_2^2-\|v\|_{\lambda_2}^2} \Big)
       \qquad\text{a.e. on } \R^n
     $$ 
     for some $k>0$. This implies $k=1$ so that $u,v$ have to be positive multiples of each
     other. From the Euler-Lagrange equation \eqref{Gl equation}  we deduce
     $\lambda_1=\lambda_2,\alpha=\beta$ which finishes the proof.  \qed
  
  \section{Proof of Theorem \ref{Thm Bifurcation}}   
   
  In this section we assume $\lambda_1,\lambda_2,\alpha,\beta>0$ as before but the space dimension $n$ is
  supposed to be $1,2$ or $3$. In Remark~\ref{Bem ngeq4} we will comment on the reason for this restriction.
  Let us first provide the functional analytical framework we will be working in. In case
  $n\geq 2$ we set $X:= H^1_r(\R^n)\times H^1_r(\R^n)$ to be the product of the radially symmetric functions
  in $H^1(\R^n)$ and define $F:X\times (0,\infty)\to X$ by
  \begin{align} \label{Gl Def F}
    F(u,v,s) &:=
    \vecII{u-(-\Delta+\lambda_1)^{-1}(\alpha u Z(1+sZ)^{-1})}{
    v - (-\Delta+\lambda_2)^{-1}(\beta vZ(1+sZ)^{-1})}  
    \quad\text{where }Z:=\alpha u^2+\beta v^2.      
  \end{align}
  Hence, finding solutions of \eqref{Gl equation} is equivalent to finding zeros of $F$.
  Using the compactness of the embeddings $H^1_r(\R^n)\to L^q(\R^n)$ for $n\geq 2$ and $2<q<\frac{2n}{n-2}$
  one can check that for all $s$ the function $F(\cdot,s)$ is a smooth compact perturbation of the identity
  in $X$ so that the Krasnoselski-Rabinowitz Global Bifurcation Theorem
  \cite{Kr_topological_methods},\cite{Rab_Some_global_results} is applicable. In case $n=1$, however, this
  structural property is not satisfied which motivates a different choice for $X$. In Appendix~A we
  show that one can define a suitable Hilbert space $X$ of exponentially decreasing functions such that
  $F(\cdot,s):X\to X$ is again a smooth compact perturbation of the identity in $X$. Besides this technical
  unconvenience the case $n=1$ can be treated in a similar way to the case $n\in\{2,3\}$ so that we only carry
  out the proofs only for the latter. Furthermore we always assume
  $\frac{\lambda_2}{\lambda_1}<\frac{\beta}{\alpha}$ according to the assumption of Theorem~\ref{Thm
  Bifurcation}.
  
  \medskip 
  
  The first step in our bifurcation analysis is to investigate the linearized problems associated to the
  equation $F(u,v,s)=0$ around the elements of the semitrivial solution branch $\mathcal{T}_2$.
  While doing this we make use of a nondegeneracy result for ground states of semilinear problems which is due
  to Bates and Shi \cite{BatShi_Existence_and_instability}. Amongst other things it tells us that $u_s$
  is a nondegenerate solution of the first equation in \eqref{Gl scalar GS}, i.e. we have the following
  result.
    
  \begin{prop} \label{Prop Nondegeneracy} 
    The linear problem
    \begin{align*}
      - \Delta \phi+ \lambda_1\phi = \frac{3\alpha^2 u_s^2+s\alpha^3 u_s^4}{(1+s\alpha u_s^2)^2} \phi,
      \qquad  \phi\in H^1_r(\R^n),\quad 0< s<\frac{\alpha}{\lambda_1}
      %, \\
      %- \Delta \psi+ \lambda_2\psi = \frac{3\beta^2 v_s^2+s\beta^3 v_s^4}{(1+s\beta v_s^2)^2} \psi,
      %\qquad \psi\in H^1_r(\R^n),\quad 0< s<\frac{\beta}{\lambda_2}
    \end{align*}
    only admits the trivial solution $\phi=0$. 
  \end{prop}
  \begin{proof} 
    In order to apply Theorem 5.4 (6) from \cite{BatShi_Existence_and_instability} we set 
    $$
      g(z):= - \lambda_1 z + \frac{\alpha^2 z^3}{1+s\alpha z^2}
      \qquad (z\in\R)
    $$ 
    so that $u_s$ is the ground state solution of $-\Delta u = g(u)$ in $\R^n$ which is centered at the
    origin. In the notation of \cite{BatShi_Existence_and_instability} one can check that $g$ is of class (A). Indeed, the
    properties (g1),(g2),(g3A),(g4A),(g5A) from page 258 in \cite{BatShi_Existence_and_instability} are satisfied for
    $b=(\frac{\lambda_1}{\alpha^2-\alpha\lambda_1 s})^{1/2},K_\infty=1$ and the unique positive number
    $\theta>b$ satisfying $(\frac{\alpha}{s}-\lambda_1)\theta^2-\frac{1}{s^2}\ln(1+s\alpha \theta^2)=0$.
    Notice that (g4A),(g5A) follow from the fact that $K_g(z):= zg'(z)/g(z)$
    decreases from $1$ to $-\infty$ on the interval $(0,b)$ and that it decreases from $+\infty$ to
    $K_\infty=1$ on $(b,\infty)$. Having checked the assumptions of Theorem 5.4 (6) we obtain that the space
    of solutions of $-\Delta \phi-g'(u_s)\phi=0$ in $\R^n$ is spanned by $\partial_1 u_s,\ldots,\partial_n
    u_s$ implying that the linear problem only has the trivial solution in $H^1_r(\R^n)$. Due to 
    \begin{equation} \label{Gl g'(us)}
      g'(u_s)= - \lambda_1 + \frac{3\alpha^2 u_s^2+s\alpha^3 u_s^4}{(1+s\alpha u_s^2)^2} 
    \end{equation}
    this proves the claim.
  \end{proof}
    
  Using this preliminary result we can characterize all possible bifurcation points on $\mathcal{T}_2$ which
  are, due to the Implicit Function Theorem, the points where the kernel of the linearized operators are
  nontrivial. For notational purposes we introduce the linear compact self-adjoint operator
  $L(s):H^1_r(\R^n)\to H^1_r(\R^n)$ for parameters $0<s<\frac{\alpha}{\lambda_1}$  by setting
  \begin{align*}    
    L(s)\phi:= (-\Delta+\lambda_2)^{-1}(W_s\phi),\qquad 
    &&\hspace{-1cm}W_s(x) := \frac{\alpha\beta u_s(x)^2}{1+s\alpha u_s(x)^2},
    &&(0<s<\frac{\alpha}{\lambda_1}) 
  \end{align*}
  for $\phi\in H^1_r(\R^n)$. Denoting by $(\mu_k(s))_{k\in\N_0}$ the
  decreasing null sequence of eigenvalues of $L(s)$ we will observe that finding bifurcation points on
  $\mathcal{T}_2$ amounts to solving $\mu_k(s)=1$ for $s\in (0,\tfrac{\alpha}{\lambda_1})$ and $k\in\N_0$. In
  fact we have the following.
  
  \begin{prop} \label{Prop kernels}
    We have  
    \begin{align*}
%      \ker(\partial_{X} F(0,v_s,s)) 
%      &= \ker(\Id-L_1(s))\times\{0\} 
%      &&\text{for }0< s<\frac{\beta}{\lambda_2}, \\
     \ker( \partial_{X} F(u_s,0,s)) 
     = \{0\}\times \ker(\Id-L(s))
     \qquad\text{for }0< s<\frac{\alpha}{\lambda_1}.      
    \end{align*}
  \end{prop}
  \begin{proof}
    For $(u,v),(\phi_1,\phi_2)\in X$ we have
    \begin{align*}
      \partial_XF_1(u,v,s)[\phi_1,\phi_2]
      &= \phi_1 -
      (-\Delta+\lambda_1)^{-1}\Big( \frac{s\alpha Z^2+3\alpha^2u^2+\alpha\beta
      v^2}{(1+sZ)^2}\phi_1+ \frac{2\alpha\beta uv}{(1+sZ)^2}\phi_2\Big) \\
      \partial_XF_2(u,v,s)[\phi_1,\phi_2]
      &= \phi_2 -
      (-\Delta+\lambda_2)^{-1}\Big( \frac{s\beta Z^2+3\beta^2v^2+\alpha\beta
      u^2}{(1+sZ)^2}\phi_2+ \frac{2\alpha\beta uv}{(1+sZ)^2}\phi_1\Big)
    \end{align*}
    Plugging in $u=u_s,v=0$ and $Z=\alpha u^2+\beta v^2=\alpha u_s^2$ gives
    \begin{align*}
      \partial_XF_1(u_s,0,s)[\phi_1,\phi_2]
      &= \phi_1 - (-\Delta+\lambda_1)^{-1}\Big( \frac{3\alpha^2
      u_s^2+s\alpha^3 u_s^4 }{(1+s\alpha u_s^2)^2}\phi_1\Big),   \\
      \partial_XF_2(u_s,0,s)[\phi_1,\phi_2]
      &= \phi_2 - (-\Delta+\lambda_2)^{-1}\Big( \frac{s\beta\alpha^2 u_s^4 +\alpha\beta
      u_s^2}{(1+s\alpha u_s^2)^2}\phi_2 \Big) \\
      &= \phi_2 - (-\Delta+\lambda_2)^{-1}\Big( \frac{\alpha\beta u_s^2}{1+s\alpha u_s^2}\phi_2 \Big) \\
      &= \phi_2 - (-\Delta+\lambda_2)^{-1}(W_s\phi_2) \\
      &= \phi_2 - L(s)\phi_2.
    \end{align*}
    From these formulas and Proposition \ref{Prop Nondegeneracy} we deduce the claim.
  \end{proof}
  
  Given this result our aim is to find sufficient conditions for the equation $\mu_k(s)=1$
  to be solvable. Since there is only few information available for any given $s>0$ our approach consists of 
  proving the continuity of $\mu_k$ and calculating the limits of $\mu_k(s)$ as
  $s$ approaches the boundary of $(0,\frac{\alpha}{\lambda_1})$. It will turn out that the limits at both
  sides of the interval exist and that they lie on opposite sides of the value~$1$ provided our sufficient
  conditions from Theorem~\ref{Thm Bifurcation} are satisfied. As a consequence these conditions and the
  Intermediate Value Theorem imply the solvability of $\mu_k(s)=1$ and it remains to add some technical
  arguments in order to apply the Krasnoselski-Rabinowitz Global Bifurcation Theorem to prove
  Theorem~\ref{Thm Bifurcation}. Calculating the limits of $\mu_k$ at the ends of
  $(0,\frac{\alpha}{\lambda_1})$ requires the Propositions~\ref{Prop us s to 0} and~\ref{Prop us,vs explode}.
  
  \begin{prop} \label{Prop us s to 0}
    We have   
    \begin{align*}
      u_s\to u_0,\qquad W_s\to \alpha\beta u_0^2 \qquad\text{as } s\to   0
      %v_s&\to v_0,\qquad V_s\to \alpha\beta v_0^2 \;\;\qquad \text{as } s\to 0
    \end{align*}    
    where the convergence is uniform on $\R^n$.
  \end{prop}
  \begin{proof}
    As in Lemma~\ref{Lem apriori} in Appendix~A one shows that on every interval $[0,s_0]$ with
    $0<s_0<\frac{\alpha}{\lambda_1}$ there is an exponentially decreasing function which bounds each of the
    functions $u_s$ with $s\in [0,s_0]$ from above. In particular, the Arzel\`{a}-Ascoli Theorem shows $u_s\to
    u_0$ and $W_s\to \alpha\beta u_0^2$ as $s\to 0$ locally uniformly on $\R^n$ so that the uniform
    exponential decay gives $u_s\to u_0$ and $W_s\to \alpha\beta u_0^2$ uniformly on $\R^n$.
  \end{proof}

  %For the calculation of the limits at $\frac{\alpha}{\lambda_1}$ we need the following result.

  \begin{prop} \label{Prop us,vs explode}
    We have
    \begin{align*}
      u_s&\to +\infty,\qquad W_s\to \frac{\beta\lambda_1}{\alpha} 
      \qquad\text{as } s\to   \frac{\alpha}{\lambda_1}
      %,\\
      %v_s&\to +\infty,\qquad V_s\to \frac{\alpha\lambda_2}{\beta} 
      %\;\qquad \text{as } s\to \frac{\beta}{\lambda_2}
    \end{align*}
    where the convergence is uniform on bounded sets in $\R^n$.
  \end{prop}
  \begin{proof}
    First we show 
    \begin{equation} \label{Gl vs explodes}
      u_s(0)=\max_{\R^n} u_s \to \infty\qquad\text{as }s\to  s^* := \frac{\alpha}{\lambda_1}.
    \end{equation}
    Otherwise we would observe $u_s(0)\to a$ for some subsequence where $a\geq 0$. In case $a>0$ a
    combination of elliptic regularity theory for \eqref{Gl scalar GS} and the Arzel\`{a}-Ascoli Theorem
    would imply that $u_s$ converges locally uniformly to a nontrivial radially symmetric function $u\in
    C^1(\R^n)$ satisfying $-\Delta u + \lambda_2 u = \frac{\alpha^2 u^3}{1+s^*\alpha u^2}$ in $\R^n$ in the
    weak sense and $u(0)=\|u\|_\infty=a$. As in Lemma~\ref{Lem apriori} we
    conclude that the functions $u_s$ are uniformly exponentially decaying so that $u$ even lies in
    $H^1_r(\R^n)$. Hence, we may test the differential equation with $u$ and obtain
    $$  
      \lambda_1 \int_{\R^n} u^2
      \leq \int_{\R^n} |\nabla u|^2 + \lambda_1 u^2
      = \int_{\R^n} \frac{\alpha^2 u^4}{1+s^*\alpha u^2}
      < \frac{\alpha}{s^*} \int_{\R^n} u^2
      = \lambda_1 \int_{\R^n} u^2
    $$
    which is impossible. It therefore remains to exclude the case $a=0$.
    In this case the functions $u_s$ would converge uniformly on $\R^n$ to the trivial solution implying that
    $u_s/u_s(0)$ would converge to a nonnegative bounded function $\phi\in C^1(\R^n)$ satisfying $-\Delta
    \phi+\lambda_1\phi=0$ on $\R^n$ and $\phi(0)=\|\phi\|_\infty=1$. Hence, $\phi$ is smooth so that
    Liouville's Theorem applied to the function $(x,y)\mapsto \phi(x)\cos(\sqrt{\lambda_1}y)$ defined on
    $\R^{n+1}$ implies that $\phi$ is constant and thus $\phi\equiv 0$ contradicting $\phi(0)=1$.
    This proves \eqref{Gl vs explodes}.
    
    \medskip
    
    Now set $\phi_s:= u_s/u_s(0)$. Using 
    $$
      -\Delta\phi_s + \lambda_1 \phi_s = \alpha\phi_s \cdot \frac{\alpha u_s^2}{1+s\alpha u_s^2}
      \qquad\text{in }\R^n
    $$
    and the fact that $\alpha u_s^2/(1+s\alpha u_s^2)$ remains bounded as $s\to s^*$ we get that the
    functions $\phi_s$ converge locally uniformly as $s\to s^*$ to some nonnegative radially nonincreasing
    function $\phi\in C^1(\R^n)$ satisfying $\phi(0)=\|\phi\|_\infty=1$. In order to prove our claim it is
    sufficient to show $\phi\equiv 1$ since this implies $u_s= u_s(0)\phi_s\to \infty$  locally uniformly and
    in particular $W_s\to \frac{\beta\lambda_1}{\alpha}$ locally uniformly.
    
    \medskip
        
    First we show $\phi>0$. If this were not true then there would exist a smallest number ${\rho\in
    (0,\infty)}$ such that $\phi|_{B_\rho}>0$ and $\phi|_{\partial B_r}=0$ for all $r\in [\rho,\infty)$. On
    $B_\rho$ we have $v_s\to\infty$ and $\alpha^2 u_s^2/(1+s\alpha u_s^2)\to\lambda_1$ implies     
    $-\Delta \phi + \lambda_1 \phi = \lambda_1\phi$ in $B_\rho$ and $\phi|_{\partial B_\rho}=0$ in
    contradiction to the maximum principle. Hence, we must have $\phi>0$ in $\R^n$. Repeating the above
    argument we find $-\Delta \phi + \lambda_1 \phi = \lambda_1 \phi$ in $\R^n$ and
    $\phi(0)=\|\phi\|_\infty=1$ so that Liouville's Theorem implies $\phi\equiv \phi(0)=1$.
  \end{proof}
  
  The previous Proposition enables us to calculate the limits of the eigenvalue functions $\mu_k(s)$ as
  $s$ approaches the boundary of $(0,\frac{\alpha}{\lambda_1})$. 
  
  \begin{prop} \label{Prop eigenvalue functions}
    For all $k\in\N_0$ the functions $\mu_k$ are positive and continuous on $(0,\frac{\alpha}{\lambda_1})$.
    Moreover we have
    \begin{align*}
      \mu_k(s) \to \bar\mu_k \quad\text{as } s\to 0, \qquad
      \mu_k(s) \to \frac{\beta\lambda_1}{\alpha\lambda_2} \quad\text{as } s\to \frac{\alpha}{\lambda_1}
      %, \\
      %\mu_2_k(s) &\to \bar\mu_2_k \quad\text{as } s\to 0,\qquad     
      %\mu_2_k(s) \to \frac{\beta\lambda_1}{\alpha\lambda_2} \quad\text{as } s\to \frac{\alpha}{\lambda_1}. 
    \end{align*}
  \end{prop}
  \begin{proof}
    As in Proposition \ref{Prop us s to 0} the uniform exponential decay of the functions $u_s$ for $s\in
    [0,s^*)$ for $s^*:=\frac{\alpha}{\lambda_1}$ implies $u_s\to u_{s_0},W_s\to W_{s_0}$ uniformly on $\R^n$
    whenever $s_0\in [0,s^*]$. Hence, the Courant-Fischer min-max-characterization for the eigenvalues
    $\mu_k(s)$ implies the continuity of $\mu_k$ as well as $\mu_k(s)\to \bar\mu_k$ as $s\to 0$.
    
    \medskip
    
    In order to evaluate $\mu_k(s)$ for $s\to s^*$ we apply Lemma~\ref{Lem
    apriori} from Appendix~C. The conditions~(i) and (ii) of the Lemma are satisfied since we have 
    $\|W_s\|_\infty=W_s(0)\to \frac{\beta\lambda_1}{\alpha}$ and $W_s\to \frac{\beta\lambda_1}{\alpha}$
    locally uniformly as $s\to s^*$ by Proposition~\ref{Prop us,vs explode}. From the Lemma we get
    $\mu_k(s)\to \frac{\beta\lambda_1}{\alpha\lambda_2}$ as $s\to s^*$ which is all we had to show.
  \end{proof}
 
   \medskip
   
  \begin{figure}[h!] \label{Fig eigenvalue functions}
    \centering
    \begin{tikzpicture}[yscale=5, xscale=2.4]
	  \draw[->] (0,0) -- (5.3,0) node[right] {$s$};
	  \draw[->] (0,0) node[below]{$0$} -- (0,1.7); 
	  \draw plot[smooth, tension=0.7] coordinates {  (0,0.8)  (3,1.3) (5,1.5)};
	  \node[left] at (0,0.8) {$\bar\mu_{k_0}$};
  	  \draw[dashed,thin] (1.15,0) node[below]{$s_{k_0}$} -- (1.15,1);
  	  \draw plot[smooth, tension=0.7] coordinates {  (0,0.5) (2,1.05) (3,0.95) (5,1.5)};
  	  \node[left] at (0,0.5) {$\bar\mu_{k_0+1}$};
  	  \draw[dashed,thin] (3.28,0) node[below]{$s_{k_0+1}$} -- (3.28,1);
  	  \draw plot[smooth, tension=0.6] coordinates {  (0,0.3) (2.6,0.5) (5,1.5)};
  	  \node[left] at (0,0.3) {$\bar\mu_{k_0+2}$};
  	  \draw[dashed,thin] (3.96,0) node[left,below]{$s_{k_0+2}$} -- (3.96,1);
  	  \draw plot[smooth, tension=0.5] coordinates {  (0,0.2)  (3,0.3) (4.3,0.9) (5,1.5)};
  	  \node[left] at (0,0.2) {$\bar\mu_{k_0+3}$};
  	  \draw[dashed,thin] (4.45,0) node[right,below]{$s_{k_0+3}$} -- (4.45,1);
  	  \draw[dashed,thin] (5,0) node[below] {$\frac{\alpha}{\lambda_1}$} -- (5,1.5);
  	  \draw[dashed,thin] (0,1.5) node[left] {$\frac{\beta\lambda_1}{\alpha\lambda_2}$} -- (5,1.5);
  	  \draw[dashed,thin] (0,1) node[left] {$1$} -- (5,1);
 	\end{tikzpicture}
 	\caption{The eigenvalue functions $\mu_{k_0},\ldots,\mu_{k_0+3}$ on $(0,\frac{\alpha}{\lambda_1})$.}
\end{figure}
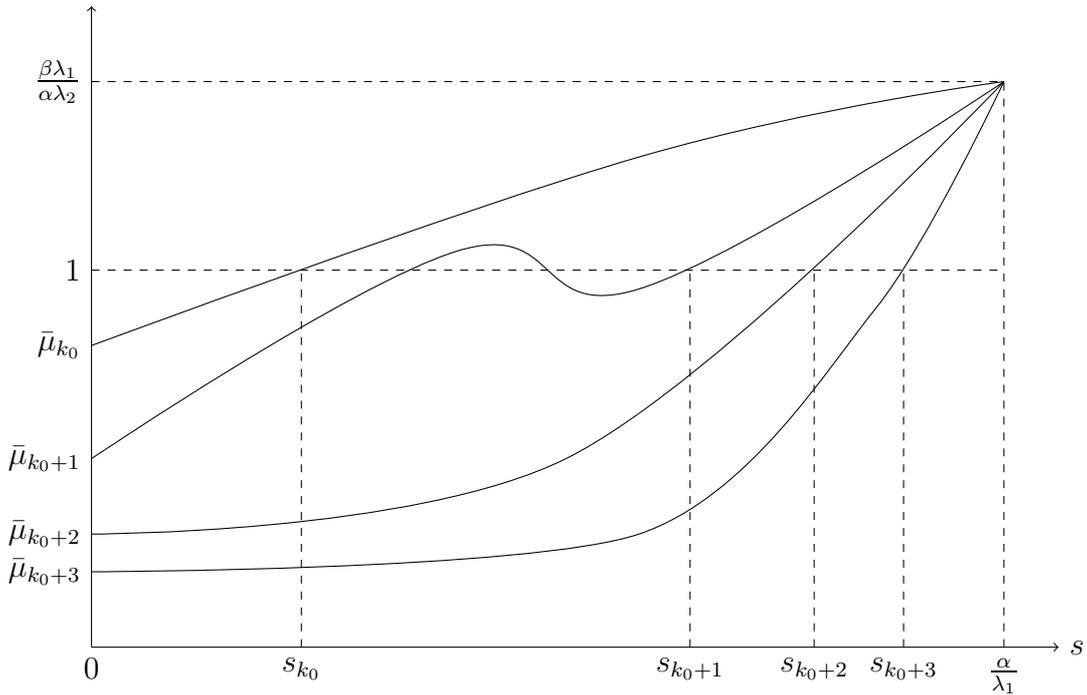

\medskip
  
  %In the following remark we clarify why our results do not apply to the case $n\geq 4$.
  
  \begin{bem} \label{Bem ngeq4}    
    When $n\geq 4$ the statement of Proposition~\ref{Prop us s to 0} is not meaningful 
    since $u_0$ does not exist in this case by Poho\v{z}aev's identity. So it is natural to ask how $u_s,W_s$
    and $\mu_k$ behave when $s$ approaches zero and $n\geq 4$. Having found an answer to this question it
    might be possible to modify our reasoning in order to prove sufficient conditions for the
    existence of bifurcation points from $\mathcal{T}_2$ in the case $n\geq 4$.
%     Given the nonexistence of nontrivial finite energy solutions of~\eqref{Gl
%     scalar GS} for $s=0,n\geq 4$ due to  there are only two possibilities: either
%     $u_s\to 0$ uniformly or $u_s\to\infty$ locally uniformly as $s\to 0$. In fact the latter occurs since
%     $u_s\to 0$ would imply that $\phi_s:=u_s/u_s(0)$ converges locally uniformly to a nontrivial bounded
%     solution $\phi\in C^1(\R^n)$ of $-\Delta\phi+\lambda_1\phi = 0$ in $\R^n$ which is impossible as we
%     showed in the proof of Proposition~\ref{Prop us,vs explode}. Hence, in case $n\geq 4$ we observe that
%     $\mu_k$ tends to $\frac{\beta\lambda_1}{\alpha\lambda_2}$ at both ends of $\mathcal{T}_2$ which makes it
%     impossible to solve the equation $\mu_k(s)=1$ using only the continuity of $\mu_k$. This is why our
%     method fails for $n\geq 4$.
  \end{bem}  
  
  The above Propositions are sufficient for proving the mere existence of the continua
  $\mathcal{C}_k$ from Theorem~\ref{Thm Bifurcation}. So it remains to show that positive solutions lie to the
  left of the threshold value $\frac{\alpha-\beta}{\lambda_1-\lambda_2}$ and that they are equibounded in
  $X$. The latter result will be proved in Lemma~\ref{Lem apriori} whereas the first claim follows from
  the following nonexistence result which slightly improves Theorem~3.10 and Theorem~3.11
  from~\cite{MMP_weakly_coupled}.
  
  \begin{prop}\label{Prop nonexistence}
    If positive solutions of \eqref{Gl equation} exist then we either have 
    $$
      \text{(i)}\quad \lambda_1=\lambda_2,\alpha=\beta
      \quad\qquad\text{or}\qquad\quad
      \text{(ii)}\quad   s < \frac{\alpha-\beta}{\lambda_1-\lambda_2}<\min\Big\{
      \frac{\alpha}{\lambda_1},\frac{\beta}{\lambda_2}\Big\}.
    $$ 
  \end{prop}
  \begin{proof}
    Assume there is a positive solution $(u,v)$ of \eqref{Gl equation}. Testing  \eqref{Gl equation} with
    $(v,u)$ leads to
    $$
      \int_{\R^n} uv\Big(\lambda_1-\lambda_2-(\alpha-\beta)\frac{Z}{1+sZ}\Big) = 0.
    $$
    Hence, the function $\lambda_1-\lambda_2-(\alpha-\beta)\frac{Z}{1+sZ}$ vanishes identically or
    it changes sign in~$\R^n$. In the first case we get (i), so let us assume that the function changes sign.
    Then we have $\lambda_1\neq \lambda_2$ and $\alpha\neq \beta$ so that Theorem~3.11 and Remark~3.18 in
    \cite{MMP_weakly_coupled} imply $0<\frac{\alpha-\beta}{\lambda_1-\lambda_2}<\min\{
    \frac{\alpha}{\lambda_1},\frac{\beta}{\lambda_2}\}$. Moreover, $s\geq
    \frac{\alpha-\beta}{\lambda_1-\lambda_2}$ would imply
    $$
      \Big|\lambda_1-\lambda_2-(\alpha-\beta)\frac{Z}{1+sZ}\Big|
      >  |\lambda_1-\lambda_2|-\frac{|\alpha-\beta|}{s}
      \geq 0
      \qquad\text{on }\R^n
    $$
    contradicting the assumption that $\lambda_1-\lambda_2-(\alpha-\beta)\frac{Z}{1+sZ}$ changes sign. Hence,
    we have $s < \frac{\alpha-\beta}{\lambda_1-\lambda_2}$ which concludes the proof.
  \end{proof}
% 
%   \begin{bem}
%      The above result is in general not true for sign-changing solutions of~\eqref{Gl equation}. Indeed, if
%      $n=1$ and $s=0,\alpha=\beta=\lambda_1=1,\lambda_2:=\omega>0$ there are nontrivial solutions $(u,v)$ given
%      by
%      \begin{align*}
%        u(x)
%        &:= \frac{2\sqrt{2} e^{x-x_0}(1+\frac{1-\omega}{1+\omega}e^{2\omega(x-x_1)})}{
%        1 + e^{2(x-x_0)} + e^{2\omega (x-x_1)} + \frac{(1-\omega)^2}{(1+\omega)^2}e^{2(x-x_0+\omega(x-x_1))}
%       }
%       \qquad (x\in\R), \\
%        v(x)
%        &:= \frac{2\sqrt{2} \omega e^{\omega(x-x_1)}(1-\frac{1-\omega}{1+\omega}e^{2(x- x_0)})}{
%        1 + e^{2(x-x_0)} + e^{2\omega (x-x_1)} + \frac{(1-\omega)^2}{(1+\omega)^2}e^{2(x-x_0+\omega(x-x_1))}
%        }\qquad (x\in\R)
%      \end{align*}
%      where $x_0,x_1\in\R$ are arbitrary. The author found these solutions in 
%      \cite{TraSip_Bound_solitary_waves}, equation (3.17).
%   \end{bem}  
%   
  
  \medskip
    
  {\it Proof of Theorem~\ref{Thm Bifurcation}:}\;
  The main ingredient of our proof is the Krasnoselski-Rabinowitz Global Bifurcation Theorem (cf.
  \cite{Kr_topological_methods},\cite{Rab_Some_global_results} or \cite{Ki_bifurcation_theory}, Theorem
  II.3.3), which, roughly speaking, says that a change of the Leray-Schauder index along a given solution
  curve over some parameter interval implies the existence of a bifurcating continuum emanating from the
  solution curve within this parameter interval. In our application the solution curve is
  $\mathcal{T}_2$ and the first task is to identify parameter intervals within $(0,\frac{\alpha}{\lambda_1})$ 
  where the index changes. For notational purposes we set $s^*:=\frac{\alpha}{\lambda_1}$.
  
  \medskip
  
  {\it 1st step: Existence of solution continua $\mathcal{C}_k$ bifurcating from $\mathcal{T}_2$.}\; 
  By assumption of the Theorem and Proposition~\ref{Prop eigenvalue functions} we have 
  $$
    \lim_{s\to 0} \mu_k(s)=\bar \mu_k<1, \qquad 
    \lim_{s\to s^*} \mu_k(s)= \frac{\beta\lambda_1}{\alpha\lambda_2} > 1
    \qquad\text{for all }k\geq k_0.
  $$ 
  The continuity of the $\mu_k$ on $(0,s^*)$ as well as $\mu_k(s)>\mu_{k+1}(s)$ for all $k\geq k_0,s\in
  (0,s^*)$ therefore implies
  $0<a_{k_0}<a_{k_0+1}<a_{k_0+2}<\ldots<\frac{\alpha}{\lambda_1}$ for the numbers $a_k$ given by 
  $$
    a_k:= \sup\Big\{ 0<s<\frac{\alpha}{\lambda_1} : \mu_k(s)<1 \Big\} \qquad (k\geq k_0).
  $$
  By definition of $a_k$ we can find $\un{a}_k<a_k<\ov{a}_k$ such that the following holds:  
  \begin{align} \label{Gl def intervals}
    \begin{aligned} 
    \text{(i)}\quad& \mu_k(s)<1<\mu_{k-1}(\un{a}_k) &&\text{for all }s\leq  \un{a}_k,k\geq k_0, \\
    \text{(ii)}\quad& \mu_k(s)>1>\mu_{k-1}(\ov{a}_k) &&\text{for all }s\geq  \ov{a}_k,k\geq k_0, \\
    \text{(iii)}\quad& a_k-\frac{1}{k}<\un{a}_k<\ov{a}_k<\un{a}_{k+1} &&\text{for all } k\geq k_0.
    \end{aligned}
  \end{align}
  In fact one first chooses $\ov{a}_k\in (a_k,a_{k+1})$ such that (ii) is satisfied and then
  $\un{a}_k<a_k$ sufficiently close to $a_k$ such that (i) and (iii) hold. Now let us show that the
  Leray-Schauder index $\ind(F(\cdot,s),(u_s,0))$ changes sign on each of the mutually disjoint intervals
  $(\un{a}_k,\ov{a}_k)$.
  
  \medskip
  
  The index of $F(\cdot,s)$ near $(u_s,0)$ is computed using the Leray-Schauder formula which involves
  the algebraic multiplicities of the eigenvalues $\mu>1$ of the compact linear operator $\Id-\partial_X
  F(u_s,0,s)$, see (II.2.11) in \cite{Ki_bifurcation_theory}. From the formulas appearing in
  Proposition~\ref{Prop kernels} we find that $\mu>1$ is such an eigenvalue if and only if one of the following equations is solvable:
  \begin{align*}
    (-\Delta+\lambda_1)^{-1}\Big( \frac{3\alpha^2 u_s^2+s\alpha^3 u_s^4}{(1+s\alpha u_s^2)^2}\phi\Big) =
    \mu\phi \quad\text{in }\R^n,\quad \phi\in H^1_r(\R^n),\phi\neq 0, \\    
    L(s)\psi=(-\Delta+\lambda_2)^{-1}(W_s\psi) = \mu\psi \quad\text{in }\R^n,\quad \psi\in 
    H^1_r(\R^n),\psi\neq 0. 
  \end{align*}
  When $s=\un{a}_k$ then the second equation is solvable with $\mu>1$ if and only if 
  $\mu$ is an eigenvalue of $L(\un{a}_k)$ larger than 1. By \eqref{Gl def intervals}~(i) this is
  equivalent to $\mu\in \{\mu_0(\un{a}_k),\ldots,\mu_{k-1}(\un{a}_k)\}$. Due to Sturm-Liouville theory each
  of these eigenvalues is simple. The first equation is solvable with $\mu>1$ if and only if
  $\Delta+g'(u_s)$ has a negative eigenvalue in $H^1_r(\R^n)$ where $g$ is defined as in \eqref{Gl g'(us)}.
  From Theorem~5.4 (4)-(6) in \cite{BatShi_Existence_and_instability} we infer that there is precisely one
  such eigenvalue $\mu>1$ and $\mu$ has algebraic multiplicity one. Denoting the
  $H^1_r(\R^n)$-spectrum with $\sigma$ we arrive at the formula
  \begin{align*}
    \ind(F(\cdot,\un{a}_k),(0,v_{\un{a}_k}))
    &= (-1)^{\#\{\mu\in\sigma(\Id-\partial_X F(0,v_{\un{a}_k},\un{a}_k)): \mu>1\}} \\
    &= (-1)^{k+1} \\
    &= -(-1)^{k+2} \\
    &= -(-1)^{\#\{\mu\in\sigma(\Id-\partial_X F(0,v_{\ov{a}_k},\ov{a}_k)): \mu>1\}} \\
    &= - \ind(F(\cdot,\ov{a}_k),(0,v_{\ov{a}_k})).    
  \end{align*}
  The Krasnoselski-Rabinowitz Theorem implies that the interval $(\un{a}_k,\ov{a}_k)$ contains a least one
  bifurcation point $(u_{s_k},0,s_k)$ so that the maximal component $\mathcal{C}_k$ in $\mathcal{S}$
  satisfying $(u_{s_k},0,s_k)\in\mathcal{C}_k$ is non-void. By Proposition~\ref{Prop kernels} this implies
  $\mu_j(s_k)=1$ for some $j\in\N_0$ and \eqref{Gl def intervals} implies $j=k$, i.e. $\mu_k(s_k)=1$.
  Indeed, property (ii) gives $\mu_{k-1}(s_k)>1$ and (i) gives $\mu_{k+1}(s_k)<1$. 
  
  \medskip
  
  {\it 2nd step: $s_k\to s^*$ as $k\to\infty.$}\; If not then we would have $s_k\to \bar
  s$ from below for some $\bar s<s^*$. From $s_k\in (\un{a}_k,\ov{a}_k)$, the inequality
  $\un{a}_k>a_k-1/k$ and the definition of $a_k$ we deduce 
  $\mu_k(t) \geq 1$ whenever $t\geq s_k+\frac{1}{k},k\geq k_0$ and thus 
  $$
    \mu_k(t)\geq 1 \qquad\text{for all } t\in \Big(\frac{\bar s+s^*}{2},s^*\Big) \text{ and }k\geq k_1
  $$
  for some sufficiently large $k_1\in\N$.
  This contradicts $\mu_k(t)\to 0$ as $k\to\infty$ for all $t\in (0,s^*)$ and the claim is proved.  
  
  \medskip
  
  {\it 3nd step: Existence of seminodal solutions within $\mathcal{C}_k$.}\; We briefly show that fully
  nontrivial solutions of \eqref{Gl equation} belonging to a sufficiently small neighbourhood of
  $(u_{s_k},0,s_k)$ are $(0,k)$-nodal. Indeed, if solutions $(u^m,v^m,s^m)$ of \eqref{Gl equation} converge
  to $(u_{s_k},0,s_k)$ then $v^m/v^m(0)$ converges to the eigenfunction $\phi$ of $L(s_k)$ with $\phi(0)=1$
  which is associated to the eigenvalue~$1$. Due to $\mu_k(s_k)=1$ and Sturm-Liouville theory $\phi$ has
  precisely $k+1$ nodal annuli so that the same is true for $v^m$ and sufficiently large $m\in\N$. On other
  hand $u^m\to u$ implies that $u^m$ must be positive for large $m$ which proves the claim. 
  
  \medskip
  
  {\it 4rd step: Positive solutions.}\; The claim concerning positive solutions of~\eqref{Gl equation} follows
  directly from Proposition~\ref{Prop nonexistence} and Lemma~\ref{Lem apriori} from Appendix~A. \qed

  \section{Proof of Corollary \ref{Cor 1} and Corollary \ref{Cor 2}}
  
  Let $\zeta\in H^1_r(\R^n)$ be the unique positive function 
  which satisfies $-\Delta\zeta +\zeta=\zeta^3$ in $\R^n$ so that $u_0,v_0$ can be rewritten as 
  $$
    u_0(x) = \sqrt{\lambda_1}\alpha^{-1} \zeta(\sqrt{\lambda_1}x),\qquad
    v_0(x) = \sqrt{\lambda_2}\beta^{-1} \zeta(\sqrt{\lambda_2}x).
  $$
  Hence, Corollary~\ref{Cor 1} follows from Theorem~\ref{Thm Bifurcation} and the estimate
  \begin{align*}
	\bar\mu_0
	= \max_{\phi\neq 0} \frac{\alpha\beta\|u_0\phi\|_2^2}{\|\phi\|_{\lambda_2}^2} 
	\leq \max_{\phi\neq 0} \frac{\alpha\beta\|u_0\|_4^2\|\phi\|_4^2}{\|\phi\|_{\lambda_2}^2} 
	= \frac{\alpha\beta\|u_0\|_4^2\|v_0\|_4^2}{\|v_0\|_{\lambda_2}^2}
	= \frac{\alpha}{\beta} \frac{\|u_0\|_4^2}{\|v_0\|_4^2} 
	= \frac{\beta}{\alpha} \Big( \frac{\lambda_1}{\lambda_2}\Big)^{\frac{4-n}{4}}.
  \end{align*}
  In case $n=1$ we have $\zeta(x)=\sqrt{2}\sech(x)$ and it is known (see for instance Lemma~5.1
  in~\cite{deFLop_Solitary_waves}) that the eigenvalue problem $\mu(-\phi''+\omega^2\phi) =
  \zeta^2\phi$ in $\R$ admits nontrivial solutions in $H^1_r(\R)$ if and only if 
  $2/\mu = (\omega+2k)(\omega+2k+1)$ for some $k\in\N_0$. This implies  
  $$
    \bar\mu_k = \frac{\beta}{\alpha}
    \frac{2}{(\sqrt{\frac{\lambda_2}{\lambda_1}}+2k)(\sqrt{\frac{\lambda_2}{\lambda_1}}+2k+1)} \qquad
    (k\in\N_0) 
  $$	
  and Corollary~\ref{Cor 2} follows from Theorem~\ref{Thm Bifurcation}.
  
%   
%   \begin{bem}
%     The Corollaries do not yield a sufficient criterion for bifurcation from $\mathcal{T}_1$ when
%     $\alpha\lambda_2<\beta\lambda_1$. Indeed, in case $n=1$ the criterion would be $\bar\mu_2^0>1$ or
%     equivalently 
%     $$
%       \frac{\beta}{\alpha} \frac{2}{\sqrt{\frac{\lambda_2}{\lambda_1}}(\sqrt{\frac{\lambda_2}{\lambda_1}}+1)}>1
%     $$
%     which is incompatible with $\alpha\lambda_2>\beta\lambda_1$. Similarly one can check that the
%     Corollaries do not provide sufficient conditions in case $n\in\{2,3\}$ either.
%   \end{bem}
  
  \section{Open problems}  \label{sec Open problems}
  
  Let us finally summarize some open problems concerning~\eqref{Gl equation} which we were not able to solve
  and which we believe to provide a better understanding of the equation. Especially the open questions
  concerning global bifurcation scenarios are supposed to be very difficult from the analytical point of
  view so that numerical indications would be very helpful, too. The following questions might be of interest:
    
  \begin{enumerate}
    \item As in the author's work on weakly coupled nonlinear Schr\"odinger systems \cite{Man_Minimal_I} 
    one could try to prove the existence of positive solutions by minimizing the Euler functional over
    the ''system Nehari manifold'' $\mathcal{M}_s$ consisting of all fully nontrivial functions $(u,v)\in X$
    which satisfy $I'(u,v)[(u,0)]=I'(u,v)[(0,v)]=0$.
    For which parameter values $\alpha,\beta,\lambda_1,\lambda_2,s$ are there such minimizers and do they
    belong to $\mathcal{C}_0$?
    \item What is the existence theory and the bifurcation scenario when
    $\alpha\lambda_2=\beta\lambda_1$ and $\alpha\neq \beta,\lambda_1\neq \lambda_2$?       
    \item In the case $\alpha=\beta,\lambda_1=\lambda_2$ the points on $\mathcal{T}_1,\mathcal{T}_2$ are
    connected by a smooth curve and the same is true for every  semitrivial solution. Do these connections
    break up when the parameters of the equation are perturbed? This is related to the question whether
    the continuum $\mathcal{C}_0$ contains $\mathcal{T}_1$.
    \item It would be interesting to know if the eigenvalue functions $\mu_k$ are strictly
    monotone. The monotonicity of $\mu_k$ would imply that $s_k$ are the only solutions of
    $\mu_k(s)=1$ so that the totality of bifurcation points is given by $(s_k)_{k\geq k_0}$. 
%     \item Do the bifurcation points $s_k$ accumulate near $\frac{\alpha}{\lambda_1}$ as the numerical studies
%     in~\cite{ChaYan_A_scalar_nonlocal},\cite{OstKiv_Multi} indicate?
    \item We expect that $\mathcal{T}_1,\mathcal{T}_2$ extend to semitrivial solution branches
    $\tilde{\mathcal{T}}_1,\tilde{\mathcal{T}}_2$ containing also negative parameter values $s$. A 
    bifurcation analysis for such branches remains open. Let us shortly comment why we expect an interesting
    outcome of such a study. In the model case $n=1$ and $\beta=\lambda_2=1$ one obtains
    from~\eqref{Gl us(0) n=1} the existence of $u_s$ for all $s<0$ as well as the enclosure $u_s(0)^2\in
    (1/(|s|+1),1/|s|)$. Using this one successively proves $su_s(0)^2\to -1$ and $s(1+su_s(0)^2)\to 0$
    %(even $|s(1+su_s(0)^2|^{1/s}\to e$) 
    as $s\to -\infty$.
%     Denn 
%     $$
%       0 
%       = \frac{1-s-\ln(|s|)}{s} + \frac{s-1}{s^2}p_s + \frac{1}{s}\ln(|p_s|)
%       = -1 + o(1) + \frac{s-1}{s^2}p_s + \frac{1}{s}\ln(|p_s|)
%       = -1 + o(1) + \ln(|p_s|^{1/s}),
%     $$
%     denn $p_s/s\to 0$, also $p_s\sim -e^{-|s|}$. 
    This implies $W_s(0) = u_s(0)^2/(1+su_s(0)^2) \to +\infty$ as $s\to -\infty$
    so that one expects $\mu_k(s)\to +\infty$ as $s\to -\infty$ for all $k\in\N_0$. In view of
    $\bar\mu_{k_0}<1$ this leads to the natural conjecture that there are infinitely many bifurcating branches
    $(\tilde{\mathcal{C}}_k)_{k\geq k_0}$ also in the parameter range $s<0$.
    \item Our paper does not contain any existence result for fully nontrivial solutions when
    $n\geq 4$ and $\lambda_1\neq \lambda_2$ or $\alpha\neq \beta$. It would be interesting to know whether
    there is a nonexistence result behind or whether an analysis following Remark \ref{Bem ngeq4} yields
    existence of fully nontrivial solutions.
  \end{enumerate}

  \section{Appendix A}
  
  In our proof of the a priori bounds for positive solutions $(u,v)$ of \eqref{Gl equation} we will use the
  notation $s^*:=\min\{\frac{\alpha}{\lambda_1},\frac{\beta}{\lambda_2}\}$ and $u(x)=\hat u(|x|),v(x)=\hat
  v(|x|)$ so that $\hat u,\hat v$ denote the radial profiles of $u,v$. Notice that all nonnegative solutions
  are radially symmetric and radially decreasing by Lemma~3.8 in~\cite{MMP_weakly_coupled}. We want to
  highlight the fact that the main ideas leading to Lemma~\ref{Lem apriori} are taken from section~2
  in \cite{Iko_Uniqueness}.
  
  \begin{lem} \label{Lem apriori}
    Let $n\in\{1,2,3\}$. For all $\eps>0$ there are $c_\eps,C_\eps>0$ such that all nonnegative solutions
    $(u,v)$ of \eqref{Gl equation} for $\lambda_1,\lambda_2,\alpha,\beta\in [\eps,\eps^{-1}]$ and $s\in
    [0,\min\{\frac{\alpha}{\lambda_1},\frac{\beta}{\lambda_2}\}-\eps]$ satisfy 
    $$
      \|u\|_{\lambda_1}+\|v\|_{\lambda_2}< C_\eps,\qquad 
      u(x)+v(x)\leq C_\eps e^{-c_\eps |x|} \quad \text{for all }x\in\R^n.
    $$
  \end{lem}
  \begin{proof}   
    {\it 1st step: Boundedness in $L^\infty(\R^n)\times L^\infty(\R^n)$.}\; Assume that there is an sequence
    $(u_k,v_k)$ of nonnegative solutions of \eqref{Gl equation} for parameters
    $(\lambda_1)_k,(\lambda_2)_k,\alpha_k,\beta_k\in [\eps,\eps^{-1}]$ and $s_k\in [0,s^*-\eps]$ which is
    unbounded in $L^\infty(\R^n)\times L^\infty(\R^n)$. As always we write $Z_k(x):=\alpha_k u_k(x)^2+\beta_k
    v_k(x)^2$. Passing to a subsequence we may assume $Z_k(0)=\max_{\R^n} Z_k \to\infty$ and
    $((\lambda_1)_k,(\lambda_2)_k,\alpha_k,\beta_k,s_k)\to (\lambda_1,\lambda_2,\alpha,\beta,s)$ for some
    $s\in [0,s^*-\eps]$ and $\lambda_1,\lambda_2,\alpha,\beta\in [\eps,\eps^{-1}]$. Let us distinguish the
    cases $s>0$ and $s=0$ to lead this assumption to a contradiction.
     
    \medskip 
    
    {\it The case $s>0$.\;} The functions 
    $$
      \phi_k:= u_k Z_k(0)^{-1/2}, \qquad \psi_k:=v_k Z_k(0)^{-1/2}
    $$ 
    are bounded in $L^\infty(\R^n)$ and satisfy $\alpha_k\phi_k(0)^2+\beta_k\psi_k(0)^2=1$ as well as
    \begin{align*}
      -\Delta\phi_k + (\lambda_1)_k\phi_k = \alpha_k \phi_k \cdot \frac{Z_k}{1+s_kZ_k} \quad\text{in }\R^n, \\
      -\Delta\psi_k + (\lambda_2)_k\phi_k = \beta_k \psi_k \cdot \frac{Z_k}{1+s_kZ_k} \quad\text{in }\R^n.
    \end{align*}
    Using $Z_k/(1+s_kZ_k)\leq s_k^{-1}=s ^{-1}+o(1)$ and deGiorgi-Nash-Moser estimates we obtain from
    the Arzel\`{a}-Ascoli Theorem that there are bounded nonnegative radially symmetric limit functions
    $\phi,\psi\in C^1(\R^n)$ satisfying $\alpha \phi(0)^2+\beta \psi(0)^2=1$ and 
    $$
      -\Delta \phi + \lambda_1 \phi = \frac{\alpha}{s} \phi\quad\text{in }\R^n,\qquad 
      -\Delta \psi + \lambda_2 \psi = \frac{\beta}{s}\psi\quad\text{in }\R^n.
    $$
    From $\lambda_1<\frac{\alpha}{s}$ and $\lambda_2<\frac{\beta}{s}$ we obtain 
    $$
      \phi(r) = \kappa_1
      r^{\frac{2-n}{2}}J_{\frac{n-2}{2}}\Big(\Big(\frac{\alpha}{s}-\lambda_1\Big)^{1/2}r\Big),\qquad 
      \psi(r) = \kappa_2
      r^{\frac{2-n}{2}}J_{\frac{n-2}{2}}\Big(\Big(\frac{\beta}{s}-\lambda_2\Big)^{1/2}r\Big) \quad\text{for }r\geq 0 $$
    for some $\kappa_1,\kappa_2\in\R$. Since the functions $\phi,\psi$ are nonnegative this is only possible
    in case $\kappa_1=\kappa_2=0$ which contradicts ${\alpha \phi(0)^2+\beta \psi(0)^2=1}$.
    Hence the case $s>0$ does not occur.
    
    \medskip
    
    {\it The case $s=0$.\;} We first show $s_kZ_k\to 0$ uniformly on $\R^n$ which, due to $Z_k(0)=\max_{\R^n}
    Z_k$, is equivalent to proving  $s_kZ_k(0)\to 0$. So let $\kappa$ be an arbitrary accumulation point of
    the sequence $(s_kZ_k(0))_{k\in\N}$ and without loss of generality we assume $s_kZ_k(0)\to \kappa\in
    [0,\infty]$ so that we are left to show $\kappa=0$. To this end set 
    $$
      \phi_k(x):= u_k(\sqrt{s_k}x) Z_k(0)^{-1/2}, \qquad \psi_k(x):=v_k( \sqrt{s_k}x) Z_k(0)^{-1/2}.
    $$ 
    The functions $\phi_k,\psi_k$ satisfy $\alpha_k \phi_k(0)^2+\beta_k\psi_k(0)^2=1$ as well as
    \begin{align*}
      -\Delta\phi_k  + s_k(\lambda_1)_k \phi_k 
      &=\alpha_k \phi_k \cdot \frac{s_kZ_k}{1+s_kZ_k} 
      =  \alpha_k \phi_k \cdot 
      \frac{s_kZ_k(0)(\alpha_k\phi_k^2+\beta_k  \psi_k^2)}{1+s_kZ_k(0)(\alpha_k\phi_k^2+\beta_k \psi_k^2)}
      \qquad\text{in }\R^n, \\
      -\Delta\psi_k + s_k(\lambda_2)_k \psi_k 
      &= \beta_k \psi_k \cdot \frac{s_kZ_k}{1+s_kZ_k}  
      = \beta_k \psi_k \cdot 
      \frac{s_kZ_k(0)(\alpha_k\phi_k^2+\beta_k  \psi_k^2)}{1+s_kZ_k(0)(\alpha_k\phi_k^2+\beta_k \psi_k^2)}
      \qquad\text{in }\R^n.
    \end{align*}
    The Arzel\`{a}-Ascoli Theorem implies that a  subsequence $(\phi_k),(\psi_k)$ converges
    locally uniformly to nonnegative functions $\phi,\psi\in C^1(\R^n)$ satisfying 
    $\alpha\phi(0)^2+\beta\psi(0)^2=1$ and
    \begin{align*}
      -\Delta\phi 
      &= \alpha \phi \cdot \frac{\kappa (\alpha \phi^2+\beta\psi^2)}{ 
        1+\kappa (\alpha \phi^2+\beta\psi^2)}\qquad\text{in }\R^n, \\
      -\Delta\psi 
      &= \beta \psi \cdot \frac{\kappa (\alpha \phi^2+\beta\psi^2)}{ 
        1+\kappa (\alpha \phi^2+\beta\psi^2)}\qquad\text{in }\R^n.
    \end{align*}
    In case $\kappa=+\infty$ we arive at a contradiction as in the case $s>0$ so let us assume
    $\kappa<\infty$. Then $z:=\phi+\psi$ is nonnegative, nontrivial and the inequality
    $\alpha\phi^2+\beta\psi^2\leq \alpha\phi(0)^2+\beta\psi(0)^2=1$ implies 
    \begin{align*}
      -\Delta z 
      &= (\alpha \phi+\beta\psi) \cdot \frac{\kappa (\alpha \phi^2+\beta\psi^2)}{ 
        1+\kappa (\alpha \phi^2+\beta\psi^2)} \\
      &\geq \min\{\alpha,\beta\}(\phi+\psi) \cdot \frac{\kappa}{1+\kappa} (\alpha \phi^2+\beta\psi^2) \\
      &\geq c(\kappa) (\phi+\psi)^3 \\
      &=c(\kappa) z^3 
    \end{align*}
    where $c(\kappa) = \min\{\alpha,\beta\}^2\kappa/(2(1+\kappa))$. From Theorem~8.4
    in~\cite{QuiSou_Superlinear_parabolic} we infer $c(\kappa)=0$ and thus 
    $\kappa=0$. Hence, every accumulation point of the sequence $(s_kZ_k(0))$ is zero so that
    $s_kZ_k$ converges to the trivial function uniformly on $\R^n$.
    
    \medskip
    
    With this result at hand one can use the classical blow-up technique by considering 
    $$
      \tilde\phi_k(x):=u_k(Z_k(0)^{-1/2}x) Z_k(0)^{-1/2},\qquad 
      \tilde\psi_k(x):=v_k(Z_k(0)^{-1/2}x) Z_k(0)^{-1/2}.
    $$ 
    These functions satisfy $\alpha_k\tilde\phi_k(0)^2+\beta_k\tilde\psi_k(0)^2=1$ as well as 
    \begin{align*}
      -\Delta\tilde\phi_k + Z_k(0)^{-1}(\lambda_1)_k \tilde\phi_k 
      &= \alpha_k \tilde\phi_k \cdot \frac{Z_k Z_k(0)^{-1}}{1+s_kZ_k}
      \qquad\text{in }\R^n, \\
      -\Delta\tilde\psi_k + Z_k(0)^{-1}(\lambda_2)_k \tilde\psi_k 
      &= \beta_k \tilde\psi_k \cdot \frac{Z_k  Z_k(0)^{-1}}{1+s_kZ_k}
      %\\
      %&= o(1) + \beta_k \psi_k \cdot Z_k Z_k(0)^{-1}  
      \qquad\text{in }\R^n.
    \end{align*}
    Then $s_kZ_k\to 0$ uniformly in $\R^n$ and similar arguments as the ones used above lead to a bounded
    nonnegative nontrivial solution $\phi,\psi$ of
    \begin{align*}
      -\Delta \phi &= \alpha \phi (\alpha\phi^2+\beta\psi^2) \qquad\text{in }\R^n,\\
      -\Delta \psi &= \beta \psi (\alpha\phi^2+\beta\psi^2) \qquad\text{in }\R^n,
	\end{align*}
    which we may lead to a contradiction as above. This finally shows that $Z_k(0)\to \infty$ is impossible
    also in case $s=0$ so that the nonnegative solutions $(u,v)$ of \eqref{Gl equation} are pointwise bounded
    by some constant depending on $\eps$.

    \medskip
    
    {\it 2nd step: Uniform exponential decay.}\; Let us assume for contradiction that there is a sequence
    $(u_k,v_k,s_k)$ of positive solutions of \eqref{Gl equation} satisfying 
    \begin{equation} \label{Gl Lem assumption}
      \hat u_k(r_k)+\hat v_k(r_k)\geq k e^{-r_k/k}\quad\text{for all }k\in\N \text{ and some }r_k>0.
    \end{equation}
    Due to the $L^\infty$-bounds for $(u_k,v_k)$ which we proved in the first step we
    can use deGiorgi-Nash-Moser estimates and the Arzel\`{a}-Ascoli theorem to obtain a smooth bounded
    radially symmetric limit function $(u,v)$ of a suitable subsequence of $(u_k,v_k)$. As a limit of positive
    radially decreasing functions $u,v$ are also nonnegative and radially nonincreasing, in particular we may define 
    $$
      u_\infty := \lim_{r\to\infty} \hat u(r) \geq 0,\quad v_\infty := \lim_{r\to\infty} \hat v(r) \geq 0.
    $$
    Our first aim is to show $u_\infty=v_\infty=0$. Since $(\hat u,\hat v)$ decreases to some limit at
    infinity we have $\hat u'(r),\hat v'(r),\hat u''(r),\hat v''(r)\to 0$ as $r\to\infty$ so that 
    \eqref{Gl equation} implies
    \begin{equation} \label{Gl Apriori bounds 2ndstep I}
      \lambda_1 u_\infty = \frac{\alpha u_\infty Z_\infty}{1+sZ_\infty},\quad 
      \lambda_2 v_\infty = \frac{\beta v_\infty Z_\infty}{1+sZ_\infty}
      \qquad\text{where }Z_\infty = \alpha u_\infty^2+\beta v_\infty^2.
    \end{equation}
    Now define 
    \begin{align*}
      E_k(r) &:= \hat u_k'(r)^2+\hat v_k'(r)^2-\lambda_1 \hat u_k(r)^2-\lambda_2\hat v_k(r)^2 +
      s^{-2}g(sZ_k(r)), \\
      E(r) &:= \hat u'(r)^2+\hat v'(r)^2-\lambda_1 \hat u(r)^2-\lambda_2\hat v(r)^2  + s^{-2}g(sZ(r)). 
    \end{align*}
    The differential equation implies $E_k'(r)=-\frac{2(n-1)}{r}(\hat u_k'(r)^2+\hat v_k'(r)^2)\leq 0$ so that
    $E_k$ decreases to some limit at infinity. The monotonicity of $\hat u_k,\hat v_k$ and
    $\hat u_k(r),\hat v_k(r)\to 0$ as $r\to\infty$ imply that this limit must be 0.
    In particular we obtain $E_k\geq 0$ and the pointwise convergence $E_k\to E$ implies that 
    $E$ is a nonnegative nonincreasing function. From this we obtain
    \begin{align*}
      0 
      &\leq \lim_{r\to\infty} E(r) 
      \;=\; -\lambda_1 u_\infty^2-\lambda_2 v_\infty^2 + s^{-2}g(sZ_\infty) \\
      &\stackrel{\eqref{Gl Apriori bounds 2ndstep I}}{=} - \frac{Z_\infty^2}{1+sZ_\infty}  +
      s^{-2}g(sZ_\infty) \;=\; \frac{1}{s^2}\Big( \frac{sZ_\infty}{1+sZ_\infty}-\ln(1+sZ_\infty)\Big).
    \end{align*}
    This equation implies $Z_\infty=0$ and hence $u_\infty=v_\infty=0$.
    
    \medskip
    
    Now let $\mu$ satisfy $0<\mu<\sqrt{\min\{\lambda_1,\lambda_2\}}$ and choose $\delta>0$.
    Due to $u_\infty=v_\infty=0$ we may choose $r_0>0$ such that $\hat u(r_0)+\hat v(r_0)<\delta/2$ holds.
    From $\hat u_k(r_0)\to \hat u(r_0),\hat v_k(r_0)\to \hat v(r_0)$ and the fact that $\hat u_k,\hat v_k$
    are decreasing we obtain $\hat u_k(r)+\hat v_k(r)\leq \delta$ for all $r\geq r_0$ and all $k\geq k_0$ for
    some sufficiently large $k_0\in\N$. Having chosen $\delta>0$ sufficiently small the inequality $\hat
    u_k',\hat v_k'\leq 0$ implies 
    $$
      - (\hat u_k+\hat v_k)'' + \mu^2 (\hat u_k+\hat v_k) \leq 0  \qquad\text{on }[r_0,\infty) 
      \text{ for all } k\geq k_0.
    $$
    Hence, the maximum principle implies that for any given $R>r_0$ the function $w_R(r):=
    e^{-\mu(r-r_0)}+e^{-\mu(R-r)}$ satisfies $\hat u_k+\hat v_k\leq w_R$ on $(r_0,R)$. Indeed, $w_R$
    dominates $\hat u_k+\hat v_k$ on the boundary of $(r_0,R)$ due to
    $$
      w_R(r_0)=w_R(R)\geq 1\geq \delta\geq (\hat u_k+\hat v_k)(r_0)
      = \max\{(\hat u_k+\hat v_k)(r_0),(\hat u_k+\hat v_k)(R)\}.
    $$
    Sending $R$ to infinity we obtain 
    $$
      (\hat u_k+\hat v_k)(r)\leq e^{-\mu(r-r_0)}\quad\text{for all }r\geq r_0.
    $$
    which, together with the a priori bounds from the first step, yields a contradiction to the
    assumption~\eqref{Gl Lem assumption}. This proves the uniform exponential decay.
    
    \medskip
    
    {\it 3rd step: Conclusion.\;} Given the uniform exponential decay of $(u,v)$ we obtain a uniform bound on
    $\|u\|_{L^4(\R^n)},\|v\|_{L^4(\R^n)}$ which, using the differential equation \eqref{Gl equation}, gives a
    uniform bound on $\|u\|_{\lambda_1},\|v\|_{\lambda_2}$. This finishes the proof.
  \end{proof}
  
  Let us mention that in view of Proposition~\ref{Prop us,vs explode} the a priori bounds from the above Lemma
  cannot be extended to the interval $s\in [0,\min\{\frac{\alpha}{\lambda_1},\frac{\beta}{\lambda_2}\}]$.
  Furthermore, positive solutions of \eqref{Gl equation} are not uniformly bounded for parameters $s$
  belonging to neighbourhoods of $0$ when $n\geq 4$, see Remark~\ref{Bem ngeq4}. Notice that the assumption
  $n\in\{1,2,3\}$ in the proof of the above Lemma only becomes important when we apply Theorem~8.4
  in~\cite{QuiSou_Superlinear_parabolic}.

  \section{Appendix B}
  
  In this section we  show that in the one-dimensional case the function $F(\cdot,s):X\to X$ given by
  \eqref{Gl Def F} is a compact perturbation of the identity for an appriately chosen Banach space $X$ such
  that $\mathcal{T}_1,\mathcal{T}_2$ are continuous curves in $X\times (0,\infty)$. Let $\sigma\in (0,1)$ be
  fixed and set $(X,\skp{\cdot}{\cdot}_X)$ to be the Hilbert space given by
  \begin{align*}
     X &:= \big\{ (u,v) \in H^1_r(\R)\times H^1_r(\R) : \skp{(u,v)}{(u,v)}_X <\infty \big\}, \\
     \skp{(u,v)}{(\tilde u,\tilde v)}_X
     &:= \int_0^\infty e^{2\sigma\mu_1 x}(u'\tilde u'+ \mu_1^2 u \tilde u)\,dx +
     \int_0^\infty e^{2\sigma\mu_2 x} (v'\tilde v'+ \mu_2^2 v\tilde v)\,dx
  \end{align*}
  where $\mu_1:=\sqrt{\lambda_1}$ and $\mu_2:=\sqrt{\lambda_2}$. One may check that $(X,\skp{\cdot}{\cdot}_X)$
  is a Hilbert space and the subspace $C_{0,r}^\infty(\R)\times C_{0,r}^\infty(\R)$ consisting of smooth even
  functions having compact support is dense in $X$. We wil use the formula
  \begin{equation} \label{Gl AppB I}
    \big((-\Delta+\mu^2)^{-1}f\big)(x) 
    = \frac{\mu}{2} \int_\R e^{-\mu|x-y|}f(y)\,dy
    = \int_0^\infty \mu \Gamma(\mu x,\mu y)f(y)\,dy
  \end{equation}
  for all $f\in C_{0,r}^\infty(\R)$ and $\mu>0$ where $\Gamma(x,y)=\frac{1}{2}(e^{-|x-y|}+e^{-|x+y|})$.
  
  \medskip 
  
  {\it Proof of well-definedness:}
  First let us prove the following estimate for all $(u,v)\in X$:
  \begin{equation} \label{Gl AppB II}
    \sqrt{\mu_1}|u(r)|\leq \|(u,v)\|_X e^{-\sigma\mu_1 r}\quad\text{and}\quad
    \sqrt{\mu_2}|v(r)|\leq \|(u,v)\|_X e^{-\sigma\mu_2 r} \qquad (r\geq 0).
  \end{equation}    
  It suffices to prove these inequalities for $u,v\in C_{0,r}^\infty(\R)$. For such functions we
  have 
  \begin{align*}
    \mu_1 u(r)^2
    &\leq 2\mu_1 \int_r^\infty |uu'|\dx
    \leq  e^{-2\sigma\mu_1 r} \int_r^\infty e^{2\sigma\mu_1 x}(u'^2 + \mu_1^2 u^2)\dx
    \leq  \|(u,v)\|_X^2  e^{-2\sigma\mu_1 r}, \\
    \mu_2 v(r)^2
    &\leq 2\mu_2 \int_r^\infty |vv'|\dx
    \leq  e^{-2\sigma\mu_2 r} \int_r^\infty e^{2\sigma\mu_2 x}(v'^2 + \mu_2^2 v^2)\dx
    \leq  \|(u,v)\|_X^2  e^{-2\sigma\mu_2 r}. 
  \end{align*}
  Next, using $u'(0)=v'(0)=0$ and that $u,v$ have compact support, we obtain
  \begin{align*}
    \int_0^\infty e^{2\sigma\mu_1 x}(u'^2+\mu_1^2 u^2)\dx 
    &= \int_0^\infty (e^{2\sigma\mu_1 x}uu')' - 2\sigma\mu_1 e^{2\sigma\mu_1 x}  uu' +  e^{2\sigma\mu_1 x}
        u(-u''+ \mu_1^2 u) \dx\\ 
    &=  - 2\sigma\mu_1 \int_0^\infty e^{2\sigma\mu_1 x}  uu'\dx +  \int_0^\infty e^{2\sigma\mu_1 x}
        u(-u''+\mu_1^2 u) \dx\\ 
    &\leq  \sigma \int_0^\infty e^{2\sigma\mu_1 x} (u'^2+\mu_1^2 u^2)\dx +  \int_0^\infty
        e^{2\sigma\mu_1 x} u(-u''+\mu_1^2 u)\dx.
  \end{align*}
  Performing the analogous rearrangements for $v$ then gives for all $u,v\in C_{0,r}^\infty(\R)$ 
  \begin{align} \label{Gl AppB III} 
    \| (u,v)\|_X^2 
    &\leq \frac{1}{1-\sigma} \int_0^\infty e^{2\sigma\mu_1 x}u(-u''+\mu_1^2 u)\dx  %\\
       % &\quad 
  		 + \frac{1}{1-\sigma}\int_0^\infty e^{2\sigma\mu_2 x}v(-v''+\mu_2^2 v)\dx.
  \end{align}
  Applying this inequality to $(u,v)=(-\Delta+\mu_1^2)^{-1}(f)\chi_R,(-\Delta+\mu_2^2)^{-1}(g)\chi_R)$ for
  $f,g\in C_{0,r}^\infty(\R)$ and a suitable family $(\chi_R)_{R>0}$ of cut-off functions converging to 1
  we obtain  
  \begin{align*} 
    &\; \Big\| \Big( (-\Delta+\mu_1^2)^{-1}(f),(-\Delta+\mu_2^2)^{-1}(g)\Big) \Big\|_X^2 \\ 
    &\stackrel{\eqref{Gl AppB III}}{\leq} \frac{1}{1-\sigma} \int_0^\infty e^{2\sigma\mu_1 x}
    (-\Delta+\mu_1^2)^{-1}(f)(x)f(x)\dx \\
  	&\;	 + \frac{1}{1-\sigma}\int_0^\infty e^{2\sigma\mu_2 x} (-\Delta+\mu_2^2)^{-1}(g)(x)g(x) \dx \\
    &\stackrel{\eqref{Gl AppB I}}{=} \frac{\mu_1}{1-\sigma} \int_0^\infty\int_0^\infty e^{2\sigma\mu_1 x} 
        \Gamma(\mu_1x,\mu_1y) f(x)f(y)\,dx\,dy \ \\
    &\;  + \frac{\mu_2}{1-\sigma} \int_0^\infty\int_0^\infty e^{2\sigma\mu_2x}
            \Gamma(\mu_2x,\mu_2 y) g(x)g(y) \,dx\,dy \\
    &\leq  \frac{\mu_1}{1-\sigma} \int_0^\infty\int_0^\infty e^{\sigma\mu_1 x}
    e^{\sigma\mu_1 y} |f(x)||f(y)|\,dx\,dy \ \\
    &\;  + \frac{\mu_2}{1-\sigma} \int_0^\infty\int_0^\infty e^{\sigma\mu_2 x} e^{\sigma\mu_2 y}
            |g(x)||g(y)| \,dx\,dy \\
    &= \frac{\mu_1}{1-\sigma} \Big(\int_0^\infty e^{\sigma\mu_1 x}|f(x)|\,dx\Big)^2 
     + \frac{\mu_2}{1-\sigma} \Big(\int_0^\infty e^{\sigma\mu_2 x}|g(x)|\,dx\Big)^2.
  \end{align*}
  Plugging in
  \begin{align*} 
    f 
    := f_{u,v} 
    := \frac{\alpha u Z}{1+sZ} 
    \leq \alpha u(\alpha u^2+\beta v^2),\qquad      
    g 
    := g_{u,v}
    := \frac{\beta vZ}{1+sZ}
    \leq \beta v(\alpha u^2+ \beta v^2)         
  \end{align*}
  and using the estimate \eqref{Gl AppB II} we find that there is a positive number $C$ depending on
  $\sigma,\mu_1,\mu_2,\alpha,\beta$ but not on $u,v$ such that
  \begin{align}\label{Gl AppB IV}
    \Big\| \Big( (-\Delta+\mu_1^2)^{-1}(f_{u,v}),(-\Delta+\mu_2^2)^{-1}(g_{u,v})\Big) \Big\|_X
    \leq C \|(u,v)\|_X^3.
  \end{align} 
  By density of $C_{0,r}^\infty(\R)\times C_{0,r}^\infty(\R)$ in $X$ this inequality also holds for $(u,v)\in
  X$. If now $(u_k,v_k)$ is a sequence in $C_{0,r}^\infty(\R)\times C_{0,r}^\infty(\R)$ converging to
  $(u,v)\in X$ then similar estimates based on \eqref{Gl AppB II} show
  \begin{align*} 
    &\;\Big\| \Big(
    (-\Delta+\mu_1^2)^{-1}(f_{u_k,v_k}-f_{u_m,v_m}),(-\Delta+\mu_2^2)^{-1}(g_{u_k,v_k}-g_{u_m,v_m})\Big)
    \Big\|_X  \\ 
    &\leq C \|(u_k-u_m,v_k-u_m)\|_X (\|(u_k,v_k)\|_X+\|(u_m,v_m)\|_X)^2
  \end{align*} 
  for some $C>0$ implying that $F:X\times
  (0,\infty)\to X$ is well-defined and that \eqref{Gl AppB IV} also holds for $(u,v)\in X$.
     
  \medskip
  
  {\it Proof of compactness of $\Id-F$:} Let now $(u_m,v_m)$ be a bounded sequence in $X$. Then we can without
  loss of generality assume ${(u_m,v_m)\wto (u,v)\in X}$ and $(u_m,v_m)\to (u,v)$ pointwise almost
  everywhere. We set 
  \begin{align*}
    f_m := \frac{\alpha u_m Z_m}{1+sZ_m},\qquad
    g_m := \frac{\beta v_m Z_m}{1+sZ_m},\qquad
    f := \frac{\alpha u Z}{1+sZ},\qquad        
    g := \frac{\beta vZ}{1+sZ}. 
  \end{align*}	  
  where $Z_m:=\alpha u_m^2+\beta v_m^2$ and $Z:=\alpha u^2+\beta v^2$. 
  Then we have $f_m\to f,g_m\to g$ pointwise almost everywhere and the estimate \eqref{Gl AppB II} implies       
  \begin{align}
        |f_m(r)|+|f(r)|
        &\leq \alpha (|u_m(r)|Z_m(r)+|u(r)|Z(r))
         \leq C (e^{-3\sigma\mu_1 r}+ e^{-\sigma(\mu_1+2\mu_2)r}), 
          \label{Bif Prop Eigenschaften Esigma Abfall fn} \\
        |g_m(r)|+|g(r)|
        &\leq \beta (|v_m(r)|Z_m(r)+|v(r)|Z(r))
         \leq C (e^{-3\sigma\mu_2 r}+ e^{-\sigma(\mu_2+2\mu_1)r})
        \label{Bif Prop Eigenschaften Esigma Abfall gn}
  \end{align} 
  for some positive number $C>0$.
  Using the estimate from above we therefore obtain
  \begin{align*}
        &\|(\Id-F)(u_m,v_m)-(\Id-F)(u,v) \|_X^2 \\
        &= \Big\| \Big((-\Delta+\mu_1^2)^{-1}(f_m-f),(-\Delta+\mu_2^2)^{-1}(g_m-g) \Big)
        \Big\|_X^2\\
        &\leq \frac{\mu_1}{1-\sigma} \Big(\int_0^\infty e^{\sigma\mu_1x} 
        |f_m(x)-f(x)|\,dx \Big)^2 + \frac{\mu_2}{1-\sigma}  \int_0^\infty  e^{\sigma\mu_2 x} 
        |g_m(x)-g(x)| \,dx\Big)^2. 
  \end{align*}      
  Using \eqref{Bif Prop Eigenschaften Esigma Abfall fn},\eqref{Bif Prop Eigenschaften Esigma Abfall gn}
  and the dominated convergence theorem we finally get 
  $$
    \|(\Id-F)(u_m,v_m)-(\Id-F)(u,v) \|_X \to 0 \qquad\text{as } m\to\infty
  $$ 
  which is all we had to show.

  \section{Appendix C}
  
  Finally we prove a spectral theoretical result which we used in the proof of Proposition~\ref{Prop
  eigenvalue functions} and for which we could not find a reference in the literature. The key
  ingredient of this result is the min-max-principle for eigenvalues of semibounded selfadjoint Schr\"odinger
  operators, see for instance Theorem~XIII.2 in~\cite{RS_methods_of4}. As
  in Proposition~\ref{Prop eigenvalue functions} we denote by $\mu_k(s)$ $(k\in\N_0)$ 
  the $k$-th eigenvalue of the compact selfadjoint operator
  \begin{equation} \label{Gl AppC Def Ls}
    L_s:H^1_r(\R^n)\to H^1_r(\R^n),\quad
    L_s\phi :=(-\Delta+\lambda)^{-1}(W_s\phi)
  \end{equation} 
  for potentials $W_s$ vanishing at infinity, i.e. $W_s(x)\to 0$ as $|x|\to\infty$. 
%   Notice that we impose the
%   latter assumption only in order to guarantee the compactness of $L_s$ which justifies the notion of an
%   ''eigenvalue function'' $\mu_k$. In a more general context one would probably have to use the notion $k$-th
%   min-max level in the sense of~Theorem~XIII.2 in~\cite{RS_methods_of4}
  
  \begin{lem}\label{Lem eigenvalues}
    Let $n\in\N$ and $\kappa,\lambda>0,a<b$, let $(W_s)_{s\in (a,b)}$ be a family of radially symmetric
    potentials $W_s:\R^n\to [0,\infty)$ vanishing at infinity and satisfying
    $$
      \text{(i)}\quad \limsup_{s\to b} \|W_s\|_\infty = \kappa\qquad\text{and}\qquad
      \text{(ii)}\quad W_s \to \kappa \text{ locally uniformly as }s\to b.
    $$
    Then for all $k\in\N_0$ we have $\mu_k(s) \to \frac{\kappa}{\lambda}$ as $s\to b$.
  \end{lem}
  \begin{proof}
    The min-max-principle and (i) imply 
    $$
      \limsup_{s\to b} \mu_k(s) 
      \leq \limsup_{s\to b} \frac{\|W_s\|_\infty}{\lambda} 
      = \frac{\kappa}{\lambda}.
    $$
    So it remains to show the corresponding estimate from below. Given the assumptions $W_s\geq 0$ and (ii) we
    find that it is sufficient to show $\mu_k^\eps\to \frac{\kappa}{\lambda}$ as $\eps\to 0$ where
    $\mu_k^\eps$ denotes the $k$-th eigenvalue of the compact self-adjoint operator $M_\eps:H^1_r(\R^n)\to
    H^1_r(\R^n)$ defined by $M_\eps\phi= (-\Delta+\lambda)^{-1}((\kappa-\eps)1_{B_{1/\eps}}\phi)$. Here,
    $1_{B_{1/\eps}}$ denotes the indicator function of the ball in $\R^n$ centered at the origin with radius
    $1/\eps$. Since $\eps\to M_\eps$ is continuous on $(0,\infty)$ with respect to the operator norm  
    the min-max characterization of the eigenvalues implies that
    $$
      \eps\mapsto \omega_k^\eps \text{ is continuous on }(0,\infty)
      \quad\text{where}\quad
      \omega_k^\eps := \frac{\kappa-\eps}{\mu_k^\eps}-\lambda.
    $$
    By definition of $\mu_k^\eps,\omega_k^\eps$ the boundary value problem 
    \begin{align*}
      &-\phi''(r) - \frac{n-1}{r}\phi'(r) = \omega_k^\eps \phi(r) &&\hspace{-4cm}\text{for } 0\leq r\leq
      \eps^{-1}, \\
      &-\phi''(r) - \frac{n-1}{r}\phi'(r) +\lambda \phi(r) = 0 &&\hspace{-4cm}\text{for } r\geq \eps^{-1}, \\
      \phi'(0)&=0,\quad \phi(r)\to 0 \;\text{ as }r\to\infty,\quad\phi\in C^1([0,\infty)). 
    \end{align*}
    has a nontrivial solution. Testing the differential equation on $[0,\eps^{-1}]$ with $\phi$ we obtain
    $\omega_k^\eps>0$. Hence, $\phi$ is given by  
    $$
      \phi(r)= \alpha\cdot  
       \begin{cases}
        c \cdot r^{\frac{2-n}{2}}J_{\frac{n-2}{2}}(\sqrt{\omega_k^\eps}\, r) &\text{if }r\leq \eps^{-1}, \\
        r^{\frac{2-n}{2}}K_{\frac{n-2}{2}}(\sqrt{\lambda}r) &\text{if }r\geq \eps^{-1}
      \end{cases} 
    $$
    for some $\alpha\neq 0$. Here, $K$ denotes the modified Bessel function of the second kind and $J$
    represents the Bessel function of the first kind. From $\phi\in C^1([0,\infty))$ we get the following
    conditions on $c$ and $\omega_k^\eps$:
    $$
      K_{\frac{n-2}{2}}(\sqrt{\lambda}\eps^{-1}) 
      = c J_{\frac{n-2}{2}}(\sqrt{\omega_k^\eps}\, \eps^{-1}),\qquad 
      \sqrt{\lambda} K_{\frac{n-2}{2}}'(\sqrt{\lambda}\eps^{-1}) 
      = \sqrt{\omega_k^\eps} c J_{\frac{n-2}{2}}'(\sqrt{\omega_k^\eps}\, \eps^{-1}).
    $$
    Due to the continuity of $\eps\to \omega_k^\eps$ on $(0,\infty)$ and due to the fact that $K$ is
    positive whereas $J$ has infinitely many zeros going off to infinity we infer that
    $\sqrt{\omega_k^\eps} \eps^{-1}$ is bounded on $(0,\infty)$. In particular this gives 
    $\omega_k^\eps\to 0$ and thus $\mu_k^\eps\to \frac{\kappa}{\lambda}$ as $\eps\to 0$
    which is all we had to show.
  \end{proof}
  
  %$\sqrt{\omega^{k,\eps}} \eps^{-1}$ zwischen NST der Bessel-Funktionen --> Konvergenzrate.
  
  \section*{Acknowledgements} 
    The work on this project was supported by the Deutsche Forschungsgemeinschaft 
    (DFG, German Research Foundation) - grant number MA 6290/2-1.

 \bibliographystyle{plain}

\begin{thebibliography}{99}
  \bibitem{AlmLie_symmetric_decreasing} Almgren, F.J., Lieb, E.: {\it Symmetric decreasing rearrangement is
  sometimes continuous}, J. Amer. Math. Soc. (1989), no.~4, 683--773.
  \bibitem{BatShi_Existence_and_instability} Bates, P.W., Shi, J.: {\it Existence and instability of spike
  layer solutions to singular perturbation problems}, J. Funct. Anal. 196 (2002), no. 2, 211–-264.
  \bibitem{ChaYan_A_scalar_nonlocal} Champneys, A.R., Yang, J.: {\it A scalar nonlocal bifurcation of
  solitary waves for coupled nonlinear Schrodinger systems}, Nonlinearity 15 (2002), 2165--2193.  
  \bibitem{deFLop_Solitary_waves} de Figueiredo, D., Lopes, O.: {\it Solitary waves for some nonlinear
  Schr\"odinger systems}, Ann. Inst. H. Poincaré Anal. Non Linéaire 25 (2008), no.
  1, 149–-161.
  \bibitem{GaSeTa_Existence_of} Gazzola, F., Serrin, J., Tang, M.: {\it Existence of ground states and free
  boundary problems for quasilinear elliptic operators}, Adv. Differential Equations 5 (2000), no. 1--3,
  1–-30. 
  \bibitem{GiNiNi_Symmetry} Gidas, B., Ni, W.M., Nirenberg, L.: {\it Symmetry of positive solutions of
  nonlinear elliptic equations in Rn. Mathematical analysis and applications}, Part A, pp. 369-–402, Adv. in
  Math. Suppl. Stud., 7a, Academic Press, New York-London, 1981.
  \bibitem{Iko_Uniqueness} Ikoma, N.: {\it Uniqueness of positive solutions for a nonlinear elliptic system},
  NoDEA Nonlinear Differential Equations Appl. 16 (2009), no. 5, 555–-567.
  \bibitem{Ki_bifurcation_theory} Kielh\"ofer, H.: {\it Bifurcation theory}, Springer, New York, 2012.
  \bibitem{Kr_topological_methods} Krasnosel'skii, M. A.: {\it Topological methods in the theory of nonlinear
  integral equations}, The Macmillan Co., New York, 1964. 
  \bibitem{MMP_positive_solutions} Maia, L.A., Montefusco, E., Pellacci, B.: {\it Positive solutions for a
  weakly coupled nonlinear {S}chr\"odinger system}, J. Differential Equations 229 (2006), no. 2, 743–-767.
  \bibitem{MMP_weakly_coupled} Maia, L.A., Montefusco, E., Pellacci, B.: {\it Weakly coupled
  nonlinear {S}chr\"odinger systems: the saturation effect}, Calc. Var. Partial Differential Equations 46 (2013),
  no.1--2, 325--351.
  \bibitem{Man_Minimal_I} Mandel, R.: {\it Minimal energy solutions for cooperative nonlinear Schr\"odinger
  systems}, Nonlinear Differential Equations and Applications NoDEA (2014), doi 10.1007/s00030-014-0281-2.
  \bibitem{McLSer_Uniqueness} McLeod, K., Serrin, J.: {\it Uniqueness of positive radial solutions of $\Delta
  u + f(u)=0$ in $\R^n$}, Arch. Rational Mech. Anal. 99 (1987), no. 2, 115–-145.
  \bibitem{OstKiv_Multi-hump} Ostrovskaya, E.A., Kivshar, Y.S.: {\it Multi-hump optical solitons in a
  saturable medium}, J. Opt. B 1 (1999), 77--83.
  \bibitem{QuiSou_Superlinear_parabolic} Quittner, P., Souplet, P.: {\it Superlinear parabolic problems.
  Blow-up, global existence and steady states}, Birkh\"auser Advanced Texts: Basler Lehrb\"ucher.
  Birkh\"auser Verlag, Basel, 2007. xii+584 pp. ISBN: 978-3-7643-8441-8.
  \bibitem{Rab_Some_global_results} Rabinowitz, P.H.: {\it Some global results for nonlinear eigenvalue
  problems}, J. Functional Analysis 7 (1971), 487–-513.
  \bibitem{RS_methods_of4} Reed, M., Simon, B.: {\it Methods of modern mathematical physics 4: Analysis of
  operators}, Academic Press, 1978.  
  \bibitem{SerTan_Uniqueness} Serrin, J., Tang, M.: {\it Uniqueness of ground states for quasilinear elliptic
  equations}, Indiana Univ. Math. J. 49 (2000), no. 3, 897–-923.   
  \bibitem{Str_existence_of_solitary} Strauss, W.A.: {\it Existence of solitary waves in higher dimensions},
  Comm. Math. Phys.~55 (1977), no.~2, 149--162.
  \bibitem{StuZho_Applying_the_mountain} Stuart, C.A., Zhou, H.S.: {\it Applying the mountain pass theorem to
  an asymptotically linear elliptic equation on $\R^n$}, Comm. Partial Differential Equations~24 (1999),
  no.~9--10, 1731-–1758.  
  %\bibitem{TraSip_Bound_solitary_waves} Tratnik, M.V., Sipe, J.E.: {\it Bound solitary waves in a
  %% birefringent optical fiber}, Phys. Rev. A 38 (1988), 2011--2017. 
\end{thebibliography}

\end{document}